\documentclass[12pt]{article} 
\usepackage{tikz,ytableau}
\usetikzlibrary{calc}\usetikzlibrary{matrix}
\usepackage{algorithm}
\usepackage[noend]{algpseudocode}
\usepackage{hyperref}

\ytableausetup{boxsize=1.5em}

\usepackage{cancel}
\usepackage[T1]{fontenc}\usepackage[utf8]{inputenc}
\usepackage{graphicx}

\usepackage{enumitem}

\newcommand\inv{{^{-1}}}

\usepackage{amsfonts,amssymb,amsthm}
\usepackage{amsmath,stmaryrd,url}
\usepackage{cleveref}

\usepackage{listings}
\lstdefinestyle{custom_python}{
  belowcaptionskip=1\baselineskip,
  frame=single,
  xleftmargin=\parindent,
  showstringspaces=false,
  keywordstyle=\bfseries,
  commentstyle=\itshape,
  tabsize=8,
  breaklines=true,
}
\newcommand\Tau{{\mathcal T}}

\newcommand\Nu{{\mathcal V}}
\newcommand\Mu{{M}}

\newcommand\Cc{{\mathcal C}}
\newcommand{\Scal}{\ensuremath{\mathcal{S}}}
\newcommand\Pc{{\mathcal P}}
\newcommand\Fc{{\mathcal F}}
\newcommand\Xc{{\mathcal X}}
\newcommand\Yc{{\mathcal Y}}
\newcommand\Dc{{\mathcal D}}
\newcommand\Rc{{\mathcal R}}
\newcommand\Tc{{\mathcal T}}
\newcommand\Lc{{\mathcal L}}
\newcommand\Ec{{\mathcal E}}

\newcommand{\ud}{\underline{d}}
\newcommand{\ualpha}{{\underline{\alpha}}}

\newcommand{\ue}{\underline{e}}

\newcommand{\ulambda}{\underline{\lambda}}

\renewcommand\div{{\operatorname{div}}}
\newcommand\supp{{\operatorname{supp}}}

\newcommand\diag{{\operatorname{diag}}}
\newcommand\Id{{\operatorname{Id}}}
\newcommand\Aff{{\operatorname{Aff}}}
\newcommand\rk{{\operatorname{rk}}}

\newcommand\Kron{{\operatorname{Kron}}}
\newcommand\LR{{\operatorname{LR}}}
\newcommand\Span{{\operatorname{Span}}}
\newcommand\Wt{{\operatorname{Wt}}}
\newcommand\Lie{{\operatorname{Lie}}}
\renewcommand\det{{\operatorname{det}}}

\newcommand\lu{{\mathfrak u}}

\newcommand\ZZ{{\mathbb Z}}\newcommand\QQ{{\mathbb Q}}
\newcommand\CC{\mathbb C}\newcommand\NN{\mathbb N}
\newcommand\longto{\longrightarrow}
\newcommand\VV{{\mathbb V}}

\newcommand\Hom{{\operatorname {Hom}}}

\newcommand\Ker{{\operatorname {Ker}}}

\newcommand\GL{\operatorname{GL}}\newcommand\SL{\operatorname{SL}}

\newcommand\Mc{{\mathcal M}}

\newtheorem{theo}{Theorem}
\newtheorem{prop}[theo]{Proposition}
\newtheorem{lemma}[theo]{Lemma}

{}
\newtheorem{exple}[theo]{{\noindent\bf Example}}{}
\newtheorem{remark}[theo]{{\noindent\bf Remark}}{}

\newcommand{\N}{\mathbb{N}}

\newcommand{\Q}{\mathbb{Q}}
\newcommand{\C}{\mathbb{C}}

\newcommand{\rank}{\mathtt{rk}}

\newcommand{\commentaire}[1]{}

\renewcommand{\iff}{\Leftrightarrow}

\newcommand{\codim}{\mbox{codim}}

\DeclareMathOperator{\Pid}{Pid}

\newcommand{\footremember}[2]{%
    \footnote{#2}
    \newcounter{#1}
    \setcounter{#1}{\value{footnote}}%
}
\newcommand{\footrecall}[1]{%
    \footnotemark[\value{#1}]%
} 

\begin{document}
\title{An Algorithm to compute the Kronecker cone and other moment cones}
\author{Roland Denis}
\author{Michael Bulois}
\author{Michael Bulois \footremember{Ste}{Université Jean Monnet, CNRS, Centrale Lyon, INSA Lyon, Universite Claude Bernard Lyon 1, ICJ UMR5208, 42023 Saint-
Etienne, France.}
\and Roland Denis\footremember{Lyon}{Universite Claude Bernard Lyon 1, CNRS, Centrale Lyon, INSA Lyon, Université Jean Monnet, ICJ UMR5208, 69622
Villeurbanne, France.} \and Nicolas Ressayre
\footrecall{Lyon}}
\maketitle

\begin{abstract}
We describe a new algorithm that computes the minimal list of inequalities for the moment cone of any representation 
of a complex reductive group, with implementation details for two fundamental cases: 
the Kronecker cone (governing the asymptotic support of Kronecker coefficients) and the fermionic cone. 
These correspond to the actions of $\GL_{d_1}(\CC)\times\cdots\times \GL_{d_s}(\CC)$ 
on $\CC^{d_1}\otimes\cdots\otimes \CC^{d_s}$ and $\GL_d(\CC)$ on $\bigwedge^r\CC^d$, respectively.
An implementation for these two cases in Python-Sage is available at \url{https://ea-icj.github.io/}.%  These two cases are implemented in Python-Sage.

Our work overcomes the fundamental limitations that previously
restricted such computations to cases like $\CC^4\otimes\CC^4\otimes\CC^4$. 
The state-of-the-art method by Vergne-Walter faced two major bottlenecks: 
one from combinatorial geometry in finite-dimensional vector spaces, 
and another from deciding whether certain dominant morphisms are birational 
- a problem in effective algebraic geometry that lacked a direct algorithmic solution.
We surmount these obstacles by: 
a novel use of Weyl group actions to master combinatorial complexity, 
and an original algorithm for deciding birationality that replaces previous workarounds relying on convex geometry.

Our approach allow us to tackle problems at a new scale.
We compute the minimal list of 5,333 (up to $\mathfrak S_3$) inequalities for the Kronecker cone 
$\CC^6\otimes\CC^6\otimes\CC^6$ in 2 hours.
 Furthermore, a parallel implementation computes the 64,792 (up to $\mathfrak S_3$) inequalities 
 for $\CC^7\otimes\CC^7\otimes\CC^7$ in 188 hours.

\end{abstract}

\tableofcontents

\section{Introduction}
\subsection{Moment cone of a representation}
\label{sec:intro_cone}

Let $G$ be a complex connected reductive group acting linearly on a vector space $\VV$, 
inducing an action on the ring of regular functions $\CC[\VV]$.  
Our focus is on describing the irreducible $G$-representations that occur in this ring.

Let $T\subset B$ be a maximal torus and a Borel subgroup of $G$. 
The finite-dimensional irreducible representations of $G$ are parametrized 
by the set $X^+(T)$ of dominant characters of $T$.
For $\lambda\in
X^+(T)$, we denote by $V_G(\lambda)=V(\lambda)$ the  irreducible representation of
highest weight $\lambda$. Set
\begin{equation}
\Scal(\VV,G):=\Scal(\VV):=\{\lambda\in X^+(T) | V(\lambda)\subset \CC[\VV]\}.
\label{eq:defSP}
\end{equation}

Actually, $\Scal(\VV)$ is a finitely generated semigroup. As a consequence, the generated cone:
\begin{equation}
\Cc(\VV,G):=\Cc(\VV):=\QQ_{\geq 0}\cdot \Scal(\VV)\label{eq:defCP}
\end{equation}
is a closed convex polyhedral cone in $\QQ X(T)$.

A fundamental example arises when $G=\GL_d(\C)^3$ acts on pairs of matrices $\VV:=\Mc_d(\C)^2$ via
$$
(g_1,g_2,g_3).(A,B)=(g_1Ag_3\inv,g_2Bg_3\inv).
$$
Here, $\Scal(\VV)$ is the set of triples
$(\lambda^1,\lambda^2,\lambda^3)$ of partitions of length at most $d$
such that the Littlewood-Richardson coefficient $c_{\lambda^1\;
  \lambda^2}^{\lambda^3}$ is nonzero (see \cite{VW:ineq} or
\cite{PR:mp}).
The cone $\Cc(\VV)$ is the celebrated Horn cone, which characterizes the
eigenvalues of triples $(A,B,C)$ of $d\times d$ Hermitian matrices
satisfying $A + B = C$.
The study of this cone has a long story, starting in 1912 with H.~Weyl
\cite{Weyl} and continuing with \cite{Horn,Bel:c=1,Kly1,Kly2,KT,KTW}. 
See also \cite{Fulton:survey} for a comprehensive survey.

\bigskip
In general, the cone $\Cc(\VV)$ is  a moment cone. Their descriptions have a long
story whose we recall some steps. Kirwan proved the convexity in
\cite{Ki:convex}.
Berenstein-Sjamaar \cite{BerSja:Mumford} obtained the first explicit
list of inequalities characterizing the cones.
Belkale-Kumar \cite{BK} reduced the list of inequalities canceling
redundant inequalities. The last author proved in \cite{GITEigen} that
the Belkale-Kumar's list of inequalities is irredundant. 
Independent approaches may be found in \cite{PR:tools,PR:mp}. 

\subsection{Aim of the paper}

This paper presents an algorithm for computing the minimal list of inequalities defining the cones $\Cc(\VV)$.
We implemented it for two cases significant in quantum physics (see \cite{prog}):

\begin{enumerate}
\item {\bf Quantum Marginal problem} or {\bf Kronecker cone}. Consider $s$ finite dimensional vector spaces $V_1,\dots,V_s$
  and their tensor product
$$
\VV=V_1\otimes\cdots\otimes V_s,
$$
endowed with the natural action of $G=\prod_{i=1}^s\GL(V_i)$.
In this case, $\Scal(\VV)$ is the set of $s$-uples
$(\lambda^{1\,*},\dots,\lambda^{s\,*})$, where the $\lambda^i$'s are  partitions satisfying $\ell(\lambda^i)\leq\dim(V_i)$
such that the multiple Kronecker coefficient $g_{\lambda^1\;\dots
 \; \lambda^s}$ is nonzero (see \cite{VW:ineq} or \cite{PR:mp}).
 Here $\lambda^*$ stands for the duality involution characterized by $V(\lambda^*)=V(\lambda)^*$.
\item {\bf Fermions}. Let $V$ be a finite dimensional vector space and
  $2\leq r\leq \dim(V)-2$ an integer. Set $\VV=\bigwedge^rV$ endowed with the natural action of $G=\GL(V)$.
\end{enumerate}

Our implementation also works for the bosonic case when $\GL(V)$ acts on $S^r V$. 
In this case, the cone $\Cc(S^r V)$ is known to be trivial (see
\cite{MacTsanov}). Hence, this case allows us to check our implementation. 

\bigskip
The problem of characterizing the set of possible reduced density matrices,
known as the quantum marginal problem in quantum information theory and as the
$N$-representability problem in quantum chemistry \cite{LNC, Ruskai}, has therefore been considered as
one of the most fundamental problems in quantum theory \cite{NAP}.

%These three examples are known as the quantum marginal problem.
In 1972, Borland and Dennis \cite{BD72} considered the fermionic cases
$\bigwedge^3\CC^6$, $\bigwedge^4\CC^7$ and $\bigwedge^4\CC^8$.
Klyachko established the first general algorithm for describing $\Cc(\bigwedge^r V)$ \cite{Kly,Kly4,AK}.
In 2004, Klyachko \cite{Kly} and Daftuar-Hayden \cite{DH} studied this
example for its interpretation in Quantum Physic and Quantum
Information Theory. % Describing $\Cc(\VV)$ is referred as the Quantum Marginal Problem and the first lists of explicit inequalities appear.
Notably, the defining inequalities of $\Cc(\VV)$ may carry physical meaning: for instance,
the Pauli principle is one of the linear inequalities bounding the polytope of a fermionic
system.

\bigskip
Moment cones have also been studied from the perspective of Geometric
Complexity Theory (GCT). For example,
\cite{BCMW}  show that the problem of deciding membership in the
moment polytope associated with a finite-dimensional unitary
representation of a compact, connected Lie group is in
NP$\cap$coNP. And the Kronecker cone $\Cc(\VV)$ plays a central role
in this work. In contrast, it has recently been shown that deciding positivity of a single Kronecker coefficient is NP-hard, in general \cite{BCMW2}.

\bigskip
{\bf Existing computation.} 
Several authors give descriptions of the Kronecker cone $\Cc(\VV)$ by explicit
linear inequalities. Yet, these descriptions are quickly
computationally infeasible: explicit descriptions are known for $(\dim V_i)_i=(3, 3,3)$ \cite{Franz}, up to
formats $(3, 3, 9)$ \cite{Kly4}, $((2)^n)=(2,\dots,2)$ \cite{HSS} and $(4,
4, 4)$ \cite{VW:ineq}. 
In the very recent preprint \cite{Wal:computingmomentpolytopes}, the cones $\Cc(\overline{G.x})$ for 
$x\in \CC^3\otimes\CC^3\otimes\CC^3$ and $\CC^4\otimes\CC^4\otimes\CC^4$ are determined. 

Experts had considered further progress to be computationally very hard due to three main obstacles:
the combinatorial explosion in hyperplane generation, limitations in convex geometry computation (e.g., with Normaliz), 
and the absence of effective algorithms for deciding birationality in algebraic geometry.

\subsection{Main contributions}

For the Horn cone, combining \cite{Kly1,Bel:c=1,KTW}, 
one gets a description of the minimal list of inequalities. 
This description has been generalized to the tensor cone of any reductive group through the work of \cite{BK} 
and \cite{GITEigen}. 
Our algorithm builds on the framework established in \cite{PR:mp} (Theorem~\ref{th:Kron} below), 
which characterizes the minimal list of inequalities defining the moment cone of a representation.

Translating this theoretical description into an efficient algorithm presents several difficulties:
\begin{enumerate}
  \item {\it Enumerate all the hyperplanes of $\QQ X(T)$} spanned by some weights of $T$ in $\VV$. 
  
The difficulty here is combinatorial explosion. Our main contribution addresses this through an 
novel use of the Weyl group action.
The algorithm proceeds in two phases: first, we identify hyperplanes $H$ whose equations $\tau$ correspond 
to dominant regular one-parameter subgroups. 
A key insight allows us to leverage the fact that when such a hyperplane contains a given weight, 
one can determine the sign of $\langle\tau,\chi\rangle$ for other weights $\chi$. 
Second, we obtain dominant nonregular one-parameter subgroups through an inductive procedure. 

%\michael{reformuler avec overcome the difficulty, using a clever way the action of the Weyl group}. 
Our use of the Weyl group action proves essential to overcome the combinatorial explosion. 
Consider the Kronecker cone Kron(5,5,5): modulo $\mathfrak{S}_3$-symmetry, 
there are 92,425,934,835 hyperplanes spanned by weights. 
Exhaustively iterating over these hyperplanes is computationally prohibitive -- a simple Python loop 
like \lstinline{for i in range(92425934835): pass} already requires about half an hour. 
We reduce the full computation to under 10 seconds for Kron(5,5,5).

\item {\it Computing general isotropy dimension.} Here, the acting group is a maximal compact subgroup of $G$.
We use an inductive algorithm based on the slice theorem.

\item  {\it Enumerating all the Weyl group elements} whose inversion sets satisfy numerical conditions. 

Here, the difficulty lies in the complexity, because the Weyl groups are big. 
To overcome this difficulty we work directly on the inversion sets. 
A key point is a Kostant's result  that characterizes the inversion sets.
%\item  Decide if a given regular map between affine spaces is dominant or not. 
%
%Here, we use a well known infinitesimal criterion. 
\item {\it Deciding Birationality.}  

We want to decide whether a given dominant regular map between affine spaces is birational. 
To solve this, we use a characterization of birationality via contraction of ramification divisors. 
This last step is the most innovative from a mathematical point of view and might be of independent interest in effective algebraic algebraic geometry. 
See Sections~\ref{sec:Step5} and \ref{sec:raf5} for more details. 

This birationality check is used to select from a redundant list of inequalities the irredundant one. 
Software of convex geometry, like Normaliz \cite{normaliz}, can in principle make this job.
However, our attempts to launch Normaliz with a redundant list of inequalities for Kron(5,5,5) proved unsuccessful due to insufficient memory.\footnote{
It is worth noting that Normaliz \emph{was} able to compute the 3,461 rays of Kron(5,5,5) 
from the minimal list of 2,730 inequalities generated by our program. 
This computation, however, required 82 hours and 60 GiB of memory on the 24 cores of a double Intel Xeon E5-2650 v4 and used 60 GiB of memory of available 128 GiB.}
Our program completes this step in about one minute. 

Beyond moment cones, our algorithm can have other applications. 
A notable example lies in computing structure constants $c_{uv}^w$ in the cohomology ring of complete 
homogeneous spaces $G/P$. Our method provides an effective procedure to determine whether these 
constants are $0,\, 1$, or greater.
\end{enumerate}

\bigskip
We also make use of a variant of the main result of \cite{BKR},  for which we provide a proof. 
The only difference from the original argument is that while \cite{BKR} uses the projectivity of the varieties 
involved, we use the assumption that $T$ has a unique closed orbit in $\VV$.

\subsection{Algorithm and results}

Our algorithm is implemented in \cite{prog} for the Kronecker, fermion and bosonic cases.  
It produces a set ${\mathcal I}$ of pairs $(\tau,w)$ where $\tau$ is a dominant one-parameter subgroup of $T$ 
and $w$ is an element of the Weyl group $W$ of $G$. 
In the output, the inequalities are the $w.\tau$'s. The computation proceeds through five sequential steps:

\begin{itemize}
  \item {\bf Step 1:} Generate an initial finite list $\Tau$ of $\tau$'s containing the list of the one-parameter subgroups 
  appearing in ${\mathcal I}$.  
  See Sections~\ref{sec:Step1} and \ref{sec:raf1}. 
  \item {\bf Step 2:} Filter some $\tau\in\Tau$ including by computing the dimension of a general isotropy. 
  See Section~\ref{sec:Step2}.
  \item {\bf Step 3:} For each remaining $\tau$, produce a set of $w$ satisfying combinatorial constraints necessary 
  to get $(\tau,w)\in{\mathcal I}$. See Sections~\ref{sec:Step3} and \ref{sec:raf3}.
  \item {\bf Step 4:} Filter some pairs requiring that some algebraic morphism $\pi_{\tau,w}$ is dominant. 
  See Sections~\ref{sec:Step4} and \ref{sec:raf4}.
  \item {\bf Step 5:} Select among the remaining pairs those for which $\pi_{\tau,w}$ is birational.  
  See Sections~\ref{sec:Step5} and \ref{sec:raf5}.
\end{itemize}

In addition to the main algorithm, we define three partial filters, namely ``BKR'' (Section~\ref{sec:BKR}), ``Linear Triangular'' (Section~\ref{sec:LinTri}) and ``Gröbner'' (Section~\ref{sec:Grobner}) --
which operate on some intermediate redundant lists of inequalities. 
These filters can either eliminate some redundant inequalities or detect some essential ones,
serving both to validate our implementation and accelerate the final refinement step.

To illustrate the algorithm's performance, we present complexity data for Kronecker 
cones with dimensions $(n,n,n)$ for $n=4,5,6$. All the data are given up to the obvious $\mathfrak S_3$-symmetry. 
A datum of the form $x\ (t\, $s$)$ means that this step produces $x$ outputs (that can be $\tau$'s or pairs $(\tau,w)$'s) in $t$ 
seconds. 
The times were obtained on a recent MacBook Air M1.
\bigskip

\setlength{\tabcolsep}{0.5\tabcolsep}
\noindent
\begin{tabular}{|c|c|c|c|c|c|c|c|}
  \hline
 dim &Step 1&Step 2&Step 3&Step 4&Step 5&Total\\
  \hline
$4\,4\,4$& 257 (0.6 s) & 42 (0.5s) &405 (0.2s) &230 (0.2s) &47 (3.8s)& 5.7s \\
  \hline
$5\,5\,5$&17\,057 (10s)&  267 (6.2s)& 9\,713 (8.4s) & 2\,261 (3.2s) & 462 (1m15s)&1m44s\\
  \hline
$6\,6\,6$&3\,369\,775&  2\,402& 277\,454 & 21\,715& 5\,333&\\
&(58 m)&(26 m)&(9m)&(1m47s)&(43m) &2h\\
\hline

\end{tabular}

\setlength{\tabcolsep}{6pt}
\bigskip
For comparison, the program \cite{VW:ineq} stops at step 4 and takes 5 min 25 s for the Kron(4,4,4) case, then delegates the redundant 
list of inequalities to a convex geometry program which extract the irredundant list in 2min. 

\bigskip 
These times are for sequential version of the implementation. 
A parallel version available in \cite{prog} reduces the Kron(6,6,6) computation to 28 minutes using the 24 cores of a double Intel Xeon E5-2650 v4 processors. It also enables us to compute the 64\,792 inequalities 
of the case Kron(7,7,7) in 188\,hours (176 for Step 1 and 12 for the others).  
Steps 2 to 5 were easy to parallelize, since they make independant tasks for each $\tau$ or 
for each pair $(\tau,w)$. 
Step 1 that runs recursively over a partial binary graph was more challenging to parallelize. 
It required deciding which nodes are worth running on independent processes. 

\bigskip
{\bf Organization of the paper.}
In Section~\ref{sec:th}, we state the fundational result for our algorithm.
In Sections~\ref{sec:Step0} to \ref{sec:Step5}, we explain the six
steps of our algorithm in a general setting. 
Only the basic algorithm is explained this first part. 
Details on certain optimizations are provided in Sections~\ref{sec:raf1} to \ref{sec:raf5} -- some general, others specific to
 Kronecker or fermionic cases.
Section~\ref{sec:3cones} presents mathematically the Kronecker, fermions and
bosons cones.  
Section~\ref{sec:exple} gives an example of inequality that can be consulted along the reading.
Sections~\ref{sec:BKR} and \ref{sec:2filters} describe three filters
that can be applied between Step~4 and Step~5. 
The first filter is based on an original adaptation of a
Belkale-Kumar-Ressayre's result \cite{BKR}, proved in Section~\ref{sec:BKR}.

\section{A theoretic description of the moment cone} 
\label{sec:th}

Our algorithm is based on a theorem that we recall here.
Fix a complex connected reductive group $G$, a maximal torus
$T$ and a Borel subgroup $B\supset T$. Let $U$ denote the unipotent
radical  of $B$, and $U^-$ its opposite.

Denote by $X_*(T)$ the group of one-parameter subgroups of $T$,
by $X_*^+(T)$ the set of dominant ones.
Recall that $X_*(T)$ is a lattice. An element $\tau\in X_*(T)$ is said
to be {\it indivisible} if $\tau=k\tau'$ for $\tau'\in X_*(T)$ and
$k\in\ZZ_{>0}$ implies $k=1$.
To any $\tau\in X_*^+(T)$ is associated a standard parabolic subgroup
$$
P(\tau)=\{g\in G\,:\,\lim_{t\to 0}\tau(t)g\tau(t\inv) {\rm\ exists\
  in\ }G\}\supset B
$$
with Levi subgroup
$$
G^\tau=\{g\in G\,:\,\forall t\in\CC^*\quad \tau(t)g=g\tau(t)\}
$$
and unipotent radical 
$$
U(\tau)=\{g\in G\,:\,\lim_{t\to 0}\tau(t)g\tau(t\inv)=e\}.
$$
Let $W=W_G$ be the Weyl group of $T$. 
For a dominant character $\lambda\in X^+(T)$, we set $\lambda^*=-w_0\lambda$, where $w_0$ is the longest element of $W_G$.
Denote by $W^{P(\tau)}\subset W$ the set of minimal length representatives of elements of $W_G/W_{G^\tau}$.

Denote by $\Phi$ the set of roots of $G$ and by $\Phi^+$ the set of
positive roots. Given $w\in W$, let $\Phi(w)=\Phi^+\cap w\inv \Phi^-$
denote the inversion set, that is the set of weights of the action of $T$ acting on
the Lie algebra $\Lie(w\inv U^-w\cap U)$. 

Let $\VV$ be  a $G$-representation. Let $\Wt(\VV)=\Wt(\VV,T)$ denote the set of
weights of the action of $T$ on $\VV$ in such a way that
$$
\VV=\oplus_{\chi\in \Wt(\VV)}\VV_\chi,
$$
where $\VV_\chi=\{v\in\VV\,:\,\forall t\in T\quad tv=\chi(t)v\}\neq\{0\}$.
For any $\chi\in\Wt(\VV)$, set $m_\chi=\dim(\VV_\chi)$ its multiplicity.
We denote by $\langle\cdot,\cdot\rangle$ the duality bracket between
$X_*(T)$ and $X(T)$. For $\tau\in X_*(T)$, set
$$
\VV^{\tau\leq 0}=\bigoplus_{\chi\in \Wt(\VV)\ s.t.\ \langle
  \tau,\chi\rangle\leq 0}\VV_\chi
\quad
{\rm and}
\quad
\VV^{\tau}=\bigoplus_{\chi\in \Wt(\VV)\ s.t.\ \langle
  \tau,\chi\rangle=0}\VV_\chi.
$$
Fix a maximal compact subgroup $K$ of $G$ such that $T\cap K$ is a
maximal torus of $K$. Set $K^\tau=G^\tau\cap K$.

\begin{theo}
  \label{th:Kron}
  We assume that the moment cone $\Cc(\VV)$ has nonempty interior in $X(T)\otimes\QQ$.
  Let $\lambda\in X^+(T)$. 

Then,
  $\lambda\in\Cc(\VV)$ if and only if
  \begin{equation}
    \label{eq:6}
    \langle w\tau, \lambda^*\rangle\leq 0,
  \end{equation}
  for any indivisible dominant one-parameter subgroup $\tau\in X^+_*(T)\setminus\{0\}$ and any $w\in W^{P(\tau)}$ such that
  \begin{enumerate}[label=(\Alph*)]
  \item\label{ass:Kron1}  the map
$$
\begin{array}{cccl}
  \pi=\pi_{\tau,w}\,:\,&(w\inv U^-w\cap U)\times \VV^{\tau\leq 0}&\longto&\VV\\
&(g,v)&\longmapsto&gv
\end{array}
$$ 
is birational;
\item\label{ass:Kron2}  
For any positive integer $l$, 
\begin{equation}
  \label{eq:wellcnum}
  \sum_{\begin{subarray}{c}
    \chi\in \Wt(\VV)\\
    \langle\tau,\chi\rangle=l 
  \end{subarray}
  }m_\chi = \sharp\{\beta\in \Phi(w)\,:\,\langle\tau,\beta\rangle=l\}.
\end{equation}
    \item \label{ass:Kron3} the general isotropy of $\Lie(K^\tau)$ on $\VV^\tau$ has real dimension one. 
    \end{enumerate}
    Moreover, this list of inequalities is irredundant, meaning that each inequality 
    defines a facet of $\Cc(\VV)$. 
\end{theo}

Like said in the introduction, the analogous Theorem~\ref{th:Kron} for the Horn cone can be deduced 
from \cite{Kly1,Bel:c=1,KTW}. Its extension to the tensor cone of any semisimple group is obtained by 
joining the main results of  \cite{BK,GITEigen}. The version presented here is proved in \cite{PR:mp}.

\section{Step 0: Compact stabilizers and nonempty interior hypothesis}
\label{sec:Step0}

In the next 6 sections, we describe our algorithm to compute the
inequalities defining the moment cones $\Cc(\VV)$. 
Our algorithm assumes (for simplicity) that $\Cc(\VV)$ has nonempty interior. This is equivalent to 
\begin{equation}\label{eq:C0}
\mbox{the general isotropy of } \Lie(K) \mbox{ on }\VV\mbox{ is trivial.}
\tag{$C_0$}
\end{equation}
Let us describe an algorithm that checks the condition~\eqref{eq:C0}. 

\label{sec:algopid}

There exists a subgroup $H\subset K$ and a dense open subset
$\Omega$ of $\VV$, such that, for any $v\in\Omega$, the stabilizer $K_v$
is conjugate to $H$. The subgroup $H$ (defined up to conjugacy) is
called the {\it principal isotropy of $\VV$}. See e.g. \cite[Section~I.5]{Dieck}.

The dimension of $H$ is also the generic dimension of the stabilizer in $\mathfrak k=\Lie(K)$. 
Call it the {\it principal isotropy dimension of $\VV$} and denote it by $\Pid(V,\mathfrak k)$.
We compute $\Pid(\VV,\mathfrak k)$ using the following well known consequence of the
slice Theorem \cite[Section~I.5]{Dieck}:

\begin{lemma}
  \label{lem:slice}
  Let $v\in \VV$. Consider $T_vK.v=\mathfrak k v$. Then 
  $$
\Pid(\VV,\mathfrak k)=\Pid(\VV/\mathfrak k v,\mathfrak k_v).
  $$
\end{lemma}

Whence the following algorithm, quickly explained in the introduction of \cite{gozzi}, which only relies on linear algebra computations.

\begin{algorithm}[H]
    \caption{Computing $\Pid(\VV,\mathfrak k)$}
    \label{algo:Kstab}
    \begin{algorithmic}[1]
     \State Initialize $\mathfrak k':=\mathfrak k$ and $\VV':=\VV$.
     \While{$\mathfrak k'$ acts nontrivially on $\VV'$}
     \State Choose a point $v\in V'$ with nontrivial orbit, 
     \State Replace $\VV'$ (resp.  $\mathfrak k'$) by $\VV'/\mathfrak k'v$  (resp. $\mathfrak k'_v$)
\EndWhile
     \Return $\Pid$=$\dim \mathfrak k'$
   
    \end{algorithmic}
\end{algorithm}

The sequence $\dim(\mathfrak k')$ is decreasing during the while loop. Hence, the algorithm ends. 
Moreover, if $v$ is choosen randomly then, with theoretic probability one, the loop is run only once. 

\section{Step 1: A first finite set of \texorpdfstring{$\tau$}{tau}}
\label{sec:Step1}

\subsection{Aim}
\label{sec:step1aim}

Step 1 consists in producing a finite list of indivisible one-parameter subgroups 
containing those appearing in Theorem~\ref{th:Kron}.  Finitness steams from the following key observation.

\begin{lemma}
If $(\tau,w)$ satisfies \ref{ass:Kron3} then it satisfies 
\begin{equation}\label{eq:3weak}
\mbox{the hyperplane }  (X(T)_{\Q})^{\tau} 
  \mbox{ is spanned by }\Wt_T(\VV)^{\tau}
\tag{C'}
\end{equation}
where $(X(T)_{\Q})^{\tau}:=\{\chi\in X(T)\otimes \QQ\,:\,\langle\tau,\chi\rangle =0\}$ 
and $\Wt_T(\VV)^{\tau}=(X(T)_{\Q})^{\tau}\cap \Wt_T(\VV)$.
\end{lemma}
\begin{proof}
Consider now the stabilizer in 
$K\cap T\subset K^{\tau}$ of a general point in $\VV^{\tau}$. 
It is the intersection, in $K\cap T$, of the Kernels of the weights in $\VV^{\tau}$.
Thus, $\codim_{X(T)\otimes \Q} \Span(\{\chi\in \Wt(\VV)\,:\, \langle\tau,\chi\rangle=0\})= \Pid(\VV^{\tau}, K\cap T)$. 
Since the one-dimensional image of $\tau$ in $T$ already stabilizes $\VV^{\tau}$, condition \ref{ass:Kron3} implies that $\Pid(\VV^{\tau}, K\cap T)=1 $ and thus that \eqref{eq:3weak} holds.
 \end{proof}

\subsection{Some useful mathematical observations}
\label{sec:step1math}
\bigskip

Consider the $\QQ$-cone $X_*^+(T)_\QQ$ generated by the dominant one-parameter subgroups. 
Our method is to consider faces of this cone one after the other.
Fix such a nonzero face $\Fc$. Let $T_\Fc$ denote the torus generated by the images of the elements of $\Fc$. 
Note that $\Fc=X_*^+(T_{\Fc})_{\QQ}$. 
Let $p: X(T)_{\QQ}\rightarrow X(T_{\Fc})_{\QQ}$ denote the corresponding natural projection. 
In particular, $p(\Wt_T(\VV))=\Wt_{T_{\Fc}}(\VV)$.

\begin{lemma} 
  \label{lem:hypTTF}
  Let $\tau\in X_*(T_{\Fc})\setminus\{0\}$ satisfying \eqref{eq:3weak} for $T$, then $\tau$ satisfies   
  \eqref{eq:3weak} with $T$ replaced by $T_{\Fc}$.
\end{lemma}
\begin{proof}
For any $\chi\in X(T)$, we have $\langle \tau, \chi\rangle=\langle \tau, p(\chi)\rangle$ by definition of $p$. 
The image by $p$ of a hyperlane is either a hyperplane or the whole space $X(T_{\Fc})_{\QQ}$. 
Since $\tau\neq 0$, $\Wt_{T_{\Fc}}(\VV)^\tau$ cannot span $X(T_{\Fc})_{\QQ}$. 
Thus, $p(\Wt_T(\VV)^{\tau})=\Wt_{T_{\Fc}}(\VV)^{\tau}$ spans a hyperplane of $X(T_{\Fc})_{\QQ}$ which has to coincide with $X(T_{\Fc})^{\tau}$. 
\end{proof}

Define the order $\leq_\Fc$ on $X(T_{\Fc})$ by 
\begin{equation}
  \label{eq:defordF}
  \chi\leq_{\Fc} \chi'\iff \forall \tau\in \Fc \quad\langle\tau,\chi\rangle\leq \langle\tau,\chi'\rangle.
\end{equation}
The notation $\chi<_{\Fc} \chi'$ means $\chi\leq_{\Fc} \chi'$ and $\chi\neq \chi'$.

Our algorithm exploits the key following observation:

\begin{lemma}
  \label{lem:ordre}
  Let $\tau$ be dominant one-parameter subgroup of $T$ that belongs to the interior of $\Fc$. 
  Then, for any $\chi$ and $\chi'$ in $X(T_\Fc)$, we have
  $$
\chi<_\Fc \chi'\implies \langle\tau,\chi\rangle<\langle\tau,\chi'\rangle.
  $$
\end{lemma}

Let $\Tau^{++}(\Fc)$ denote the set of indivisible $\tau$ in the interior of $\Fc$ satisfying 
\eqref{eq:3weak} for $T_{\Fc}$. 
Any $\tau\in \Tau^{++}(\Fc)$ induces a partition of $\Wt_{T_\Fc}(\VV)$ into three subsets 
$S_-^\tau$, $S_0^\tau$ and $S_+^\tau$  according to the vanishing and the sign of $\langle \tau,\chi\rangle$. 
Condition \eqref{eq:3weak} for $T_{\Fc}$ states that $S_0^{\tau}=\Wt_{T_{\Fc}}(\VV)^{\tau}$ spans a hyperplane of $X(T_{\Fc})_{\QQ}$.

For each fixed face $\Fc$, the algorithm~\ref{algo:Sreg} below produces a collection of subsets $S^1,\dots,S^s$ of $\Wt_{T}(\VV)$ such that
$$
\forall \tau\in \Tau^{++}(\Fc)\qquad \exists i \quad S_0^\tau=S^i.
$$
Our algorithm exploits the fact that if $\chi_0\in S_0^\tau$ then 
\begin{itemize}
  \item $\forall \chi'<_\Fc\chi\qquad \chi'\in S_-^\tau$;
  \item $\forall \chi'>_\Fc\chi\qquad \chi'\in S_+^\tau$.
\end{itemize}
So each time a given character is put in a candidate hyperplane $S_0$ we can decide the position of some 
other ones relatively to the partition. 

Moreover, Lemma~\ref{lem:hypTTF} allows us to work with weights of $T_\Fc$ 
in place of weights of $T$. 

\subsection{Algorithm}

The following pseudo-code constructs recursively the sets $S_0, S_+,S_-, S_{\neq0}$ and $S_{?}$ partitioning $\Wt_{T_{\Fc}}(\VV)$, with a selection of $S_0$ each time it spans a hyperplane.
Here at each step, $S_0$ is the set of weights we decided to put in the hyperplane, $S_{\pm}$ are the weights known 
to be on one side of the hyperplane, $S_{\neq0}$ are the weights only known to not belong to the hyperplane 
and $S_?$ are the indeterminate ones.

\begin{algorithm}[H]
    \caption{Computation of candidates orthogonal to $\tau\in\Tau^{++}(\Fc)$}
    \label{algo:Sreg}
    \begin{algorithmic}[1]
       \Procedure{RECURS\_S($S_0$ ; $S_+$ ; $S_-$ ; $S_{\neq0}$ ; $S_{?}$)\label{proc:rec}}{}  
      \If{$\codim \Span (S_0)=1$}\label{line:Cp}
       \Return{$\mathcal S=\{S_0\}$} \Comment{$S_0$ is a candidate}
     \ElsIf{$S_?=\emptyset$ or $\sharp( S_0\cup S_{?})<\dim T-1$}
      \Return{$\mathcal S=\{\}$} 
      \Statex\Comment{No hyperplane can be generated by some $S'$ with $S_0\subset S'\subset S_0\cup S_?$}       
    \Else
        \State{Choose $\chi\in S_{?}$} \label{line:choose_chi}
       \State $\mathcal S$=RECURS\_S($S_0$ ; $S_+$ ; $S_-$ ; $S_{\neq0}\cup\{\chi\} $ ;  $S_?\setminus\{\chi\}$)   
       \Statex\Comment{We explore the branch where $\chi \notin S_0$}    
\State{\label{line:lem:ordre}Transfer weights satisfying $\chi'>\chi$ (resp. $\chi'<\chi$) from $S_?$ or $S_{\neq0}$ to $S_+$ (resp. to $S_-$) }     
\Statex\Comment{When $\chi\in S_0$, we apply Lemma~\ref{lem:ordre}}
     %\If{$\sum_{\chi \in S_+}m_\chi\leqslant u_\Fc$}  \label{line:uF}
    \State\Return{$\mathcal S\cup$ RECURS\_S($S_0\cup\{\chi\}$ ; $S_+$ ; $S_-$ ; $S_{\neq0} $ ; $S_?\setminus\{\chi\}$)} 
    \Statex\Comment{We explore the branch where $\chi \in S_0$}    
    \EndIf
    
     \EndProcedure 
       \State{$\mathcal S_{reg}$=RECURS\_S($S_0=\emptyset$ ; $ S_+=\emptyset$ ; 
       $S_-=\emptyset$ ; $S_{\neq0}=\emptyset$ ; $S_{?}=\Wt_{T_\Fc}(\VV)$)}
    \end{algorithmic}
\end{algorithm}

In order to compute $\Tau^{++}(\Fc)$, we proceed as follows. For each produced $S_0$ generating a hyperplane, we 
compute the two indivisible generators $\pm \tau$ of the orthogonal of $S_0$. Then we select a given one-parameter subgroup $\tau$ only if:
\begin{itemize}
\item $\tau$ is dominant and belongs to the interior of $\Fc$.
\item $\tau$ does not appear as a previously computed one-parameter 
subgroup (different $S_0$ might produce the same hyperplane).
\item $\tau$ satisfies condition \eqref{ass:Kron3} for $T$ (the computation made 
on line \ref{line:Cp} only checks the weaker condition \eqref{eq:3weak} for $T_{\Fc}$).
\end{itemize}

\section{Step 2: Filtering the list of \texorpdfstring{$\tau$}{tau}}
\label{sec:Step2}

Thanks to Step~1, we are now left with a finite set of dominant  $\tau$, 
candidates to appear in inequalities of Theorem~\ref{th:Kron}.
In this step, we filter more precisely those candidates by checking an intermediate condition \eqref{eq:2weak}
 and full Condition~\ref{ass:Kron3}. 
 
The following lemma is clear:
\begin{lemma}
  \label{lem:dim}
  Condition~\ref{ass:Kron2} implies that $\VV^{\tau>0}$ is isomorphic
  as a $\tau(\CC^*)$-module to some submodule of $\Lie(U(\tau))$: for any positive $l$
\begin{equation}  
\label{eq:2weak}
  \sharp\{\chi\in\Wt(\VV)\,:\,\langle\chi,\tau\rangle=l\}\leq
  \sharp\{\beta\in\Phi^+\,:\,\langle\beta,\tau\rangle=l\}. \tag{B'}
  \end{equation}
\end{lemma}

Note that Condition \eqref{eq:2weak} depends only  on $\tau$ and is easily checked combinatorially for each $\tau$.

Condition \ref{ass:Kron3} is then checked by applying Algorithm \ref{algo:Kstab} to $\VV^{\tau}$ and $\mathfrak k^{\tau}$.

\section{Step 3: Generate the elements \texorpdfstring{$w\in W^{P(\tau)}$}{WP} satisfying
  \texorpdfstring{Condition~\ref{ass:Kron2}}{(Condition ass:Kron2)}}
\label{sec:Step3}

From the previous steps, we have selected one-parameter subgroups $\tau$ satisfying Conditions~\ref{ass:Kron3}
 and \eqref{eq:2weak}. For each such $\tau$, we now list the $w\in W^{P(\tau)}$ such that the pair
$(\tau,w)$ satisfies the full condition~\ref{ass:Kron2} of Theorem~\ref{th:Kron}. 

The naive path consisting of enumerating all the elements of $W^{P(\tau)}$ is time consuming. \emph{e.g.}, for $d=(6,6,6)$ there are $(6!)^3\simeq 3 \times 10^8$ elements to consider in $W^{P(\tau)}$ for each regular $\tau$. A simple basic enumeration of the form \lstinline{for ($\tau$,$w$): pass} already requires more than an hour in Python for this case, time to be compared to the $9$ minutes spent by our program. 

In our implementation, we follow another approach. An element $w$ is characterized by its inversion set $\Phi(w)$ ($\subset \Phi^+$) which is subject to strong constraints:

\begin{enumerate}
  \item $\Phi(w)$ is convex  
 (that is, if $\alpha,\beta\in \Phi(w)$ and $\alpha+\beta\in \Phi$ then $\alpha+\beta\in \Phi(w)$).
  \item $\Phi(w)$ is coconvex in $\Phi^+$
  (that is, for any $\alpha,\beta\in \Phi^+$ such that $\alpha,\beta\not\in \Phi(w)$ and $\alpha+\beta\in \Phi$ then $\alpha+\beta\not\in \Phi(w)$).

  \item If $\alpha\in \Phi(w)$, $\beta\in \Phi^{-}(G^{\tau})$ and $\alpha+\beta \in \Phi^+$ then $\alpha+\beta\in \Phi(w)$.
  \item For any positive integer $l$ the cardinality of $\Phi(w)^{\tau=l}:=\{\alpha\in\Phi(w)\,:\,\langle\tau,\alpha\rangle=l\}$ 
  is determined by $\dim \VV^{\tau=l}$, the LHS of condition \eqref{eq:wellcnum} (which does not depend on $w$).
\end{enumerate}
The first two statements are consequences of the (co-)convexity of $w^{-1}U^{-}w$ in $\Phi$. The third one follows from the hypothesis $w\in W^{P(\tau)}$.
The fact that these conditions exactly characterize the inversion sets of the elements of $W^{P(\tau)}$ is \cite[Proposition~5.10]{Kost}.

These conditions can be used to perform a tree-based search on the set of possible inversion sets $\Phi(w)$. We sketch such a search in the algorithm \ref{algo:W} below. Here at each step, $\Phi_y$ (resp. $\Phi_n$) is the set of roots we assume to lie (resp. to not lie) in a candidate $\Phi(w)$ while $\Phi_?$ are the indeterminate ones.

\begin{algorithm}[H]
    \caption{Computation of candidates $\Phi(w)$}
    \label{algo:W}
    \begin{algorithmic}[1]
       \Procedure{RECURS\_PHI($\Phi_y$ ; $\Phi_n$ ; $\Phi_{?}$)\label{proc:recphi}}{}  
      \If{$\Phi_y$ satisfies constraints 1 to 4}\label{line:Cp_Phi}
       \Return{$\mathcal S=\{\Phi_y\}$} 
    \Statex\Comment{$\Phi_y$ is a possible inversion set}
     \ElsIf{$\Phi_?=\emptyset$ or $\sharp (\Phi_y)_{\tau=l}>\dim \VV_{\tau=l}$ or $\sharp (\Phi_y)_{\tau=l}+\sharp (\Phi_?)_{\tau=l}<\dim \VV_{\tau=l}$}
      \Return{$\mathcal S=\{\}$} 
      \Statex\Comment{No inversion set $\Phi'$ with $\Phi_y\subset \Phi'\subset \Phi_y\cup \Phi_?$ satisfies condition 4}       
    \Else
        \State{Choose $\alpha\in \Phi_{?}$} \label{line:choicealpha}
       \State{Set $\Phi_n'=\Phi_n\cup\{\alpha\} $} \Comment{We explore the branch where $\alpha \notin \Phi_y$}
       \State{Set $\Phi_n'=\Phi_n'\cup ((\alpha+\Phi_n)\cap \Phi^+)$}\Comment{We make use of constraint 2} \label{line:phin}
       \State $\mathcal S$=RECURS\_PHI($\Phi_y$ ; $\Phi_n' $ ;  $\Phi_?\setminus \Phi_n'$)
       \State{Set $\Phi_y'=\Phi_y\cup\{\alpha\} $} \Comment{We explore the branch where $\alpha \in \Phi_y$}
       \State{Set $\Phi_y'=\Phi_y'\cup ((\alpha+\Phi_y)\cap \Phi^+)\cup((\alpha+\Phi^{-}(G^{\tau}))\cap \Phi^+)$ } \label{line:phiy}
       \Statex\Comment{We make use of constraints 1 and 3}
       \State \Return{$\mathcal S\cup$ RECURS\_PHI($\Phi_y'$ ; $\Phi_n $ ;  $\Phi_?\setminus \Phi_y'$)}
    \EndIf
     \EndProcedure 
       \State{$\mathcal S_{reg}$=RECURS\_PHI($\Phi_y=\emptyset$ ; $ \Phi_n=\emptyset$ ; 
       $\Phi_{?}=\Phi^+$)}
    \end{algorithmic}
\end{algorithm}

%Including or not a positive root in $\Phi(w)$ has consequences on the other 
%The fourth condition is the key to perform a tree-based search. Namely, at eachselecting  elements of $\Phi^+$ in candidates $\Phi(w)$ one by
We explain in Section \ref{sec:raf3} how constraints 1-3 can be interpreted in the type A case in order to handle efficiently the combinatorial operations of lines \ref{line:phin} and \ref{line:phiy} and in which order choosing the $\alpha$'s (line \ref{line:choicealpha}).
\section{Step 4: Select dominant \texorpdfstring{$\pi$}{pi}}
\label{sec:Step4}

\subsection{Aim}
From the obtained pairs $(\tau,w)$ satisfying Conditions~\ref{ass:Kron2} and \ref{ass:Kron3}  of
Theorem~\ref{th:Kron},
we now want to keep those satisfying the
following weaker version of \ref{ass:Kron1}:

\begin{equation}
  \label{eq:pidom}
  \pi_{\tau,w}\mbox{ is dominant.} \tag{A'}
\end{equation}

\subsection{Mathematical background}

Given a morphism $f\,:\,X\longto Y$ and a point $x\in X$, we denote by
$T_xf\,:\,T_xX\longto T_{f(x)}Y$ the tangent map between the tangent
spaces. The following statement is well known (see e.g. \cite[\S VII.5]{AltK}), 

\begin{lemma}
  \label{lem:domTpi}
  Let $f\,:\,X\longto Y$ be a morphism between two 
   irreducible smooth
  complex varieties of the same dimension.
  Set
  $$
  X^\circ_f:=\{x\in X\,:\,
  %\begin{array}{l}
 %   x\mbox{ is smooth, and}\\
T_xf \mbox{ is invertible}
%  \end{array}
\}.
  $$
  Then
  \begin{enumerate}
  \item $X^\circ_f$ is open in $X$;
    \item $X^\circ_f$ is nonempty if and only if $f$ is dominant. 
  \end{enumerate}
\end{lemma}

Let $\tau$ be a dominant one-parameter subgroup and $w\in
W^{P(\tau)}$.
Set $\lu(\Phi(w)):=\Lie(w\inv U^-w\cap U)$. 
For any positive $\ell$, define $\VV^{\tau=\ell}$ to be $\{v\in \VV\,:\,
\tau(t)v=t^{\ell}v\quad\forall t\in\CC^*\}$, and $\lu(\Phi(w))^{\tau={\ell}}$
similarly.

\begin{lemma}
  \label{lem:wellc}
  Let $\tau$ be a dominant one-parameter subgroup and $w\in
  W^{P(\tau)}$ such that Condition~\ref{ass:Kron2} of Theorem~\ref{th:Kron} holds and let $\pi=\pi_{\tau,w}$ be as in Theorem~\ref{th:Kron}.
Then
\begin{enumerate}
\item \label{ass:wellc1}
  $\pi$ is dominant if and only if there exists $x\in \VV^\tau$
  such that $T_{(e,x)}\pi$ is invertible.
\item \label{ass:wellc2} For any $x\in \VV^\tau$, the linear map $T_{(e,x)}\pi$ is block
  diagonal. More precisely, it satisfies
  $$
T_{(e,x)}\pi(\lu(\Phi(w))^{\tau=\ell})\subset \VV ^{\tau=\ell}\quad\forall
\ell>0
$$
and, it restricts to $\VV^{\tau\leq 0}$ as the identity.
\end{enumerate}
\end{lemma}

\begin{proof}
We have an action of the group $U(\Phi(w)):=(w^{-1}U^{-}w\cap U)$ on $(U(\Phi(w))\times \VV^{\tau\leq 0}$  by left multiplication on the first factor. It also acts naturally on $\VV$. With respect to these two actions, $\pi$ is $U(\Phi(w))$-equivariant. 
In particular, the geometry of $\pi$ at a point $(u,x)$ is isomorphic to the geometry of $\pi$ at the point $(e, u^{-1}x)$.

The differential map at $(e,x)$ satisfies $T_{(e,x)}\pi (\mathfrak u(\Phi(w))\times \VV^{\tau\leq0})=\mathfrak u(\Phi(w))\cdot x+\VV^{\tau\leq0}$. 
Moreover, the weight decomposition restricted to $\tau$ yields that for $x_{\ell}\in \VV^{\tau=\ell}$, we have  $\lu(\Phi(w))^{\tau=\ell'}\cdot x_{\ell}\subset \VV^{\tau=\ell+\ell'}$.

Now let $x=\sum_{\ell\leq0} x_{\ell}\in V^{\tau\leq0}$ with each $x_{\ell}\in \VV^{\tau=\ell}$. From hypothesis \ref{ass:Kron2} we see that $T_{(e,x)}\pi$ is a block triangular map whose diagonal square blocks represent the linear maps $\lu(\Phi(w))^{\tau=\ell'}\rightarrow\VV^{\tau=\ell'}$, $u\mapsto u\cdot x_0$ for $\ell'>0$ with one extra block corresponding to the identity map on $V^{\tau\leq 0}$.
In particular, $T_{(e,x)}\pi$ is invertible if and only if each diagonal square block is invertible.
\end{proof}

\subsection{Application to the algorithm}\label{subsec:algo_dom}

In order to select the pairs $(\tau,w)$ such that $\pi_{\tau,w}$ is dominant, we check if $\pi$ satisfies Assertion~\ref{ass:wellc1} of
Lemma~\ref{lem:wellc}.
The second assertion of this lemma allows us to make this checking more
efficient.

To check generic invertibility of $T \pi$ can use two methods:
\begin{itemize}
\item {\bf Probabilistic.}
  Pick a random point $x_0$ in $\VV^\tau$ and compute the rank $r_0$ of
  $T_{x_0}\pi$. With theoretic probability one, $x_0$ lie in the open subset where $r_0$ is  the maximal possible rank of
  $T_x\pi$. Hence, with theoretic probability
  one, $\pi$ is dominant if and only if $r_0=\dim(\VV)$.

  Assertion~\ref{ass:wellc2} of Lemma~\ref{lem:wellc} allows us to speed
  up the checking of the invertibility of the matrix of $T_{x_0}\pi$ by checking ranks of smaller matrices.

  Observe that whenever this method finds an invertible differential at some point, then $\pi$ is indeed dominant. However, we may  encounter points where the differential has lower rank with $\pi$ is dominant
  nevertheless.

In order to maximize the reliability of this method in our computation, we launched the computation several times
for each pair $(\tau,w)$. 

Here, randomness has been realized by choosing coordinates randomly evenly in $\pm[\![1, 10 000 ]\!]\pm I [\![1, 10 000 ]\!]$ in a given basis.

 \item {\bf Symbolic.} Here we make the same thing as in the
probabilistic method with $x_0$ the generic point of $\VV^\tau$. 
The computation of $T_{x_0}\pi$ and its rank are done in the rational function field $\QQ(V)$. 
Now, we get
   $$
   \rank(T_{x_0}\pi)=\dim(\VV)\iff
   \pi\mbox{ is dominant.}
   $$
   This method is deterministic but slower. For example, in symbolic this steps takes 5s (vs 0.4s in probabilistic) 
   for Kron(4,4,4) 
   and 1h13 (vs 1s) for Kron(5,4,4).
\end{itemize}

\section{Step 5: Select birational \texorpdfstring{$\pi$}{pi}}
\label{sec:Step5}

\subsection{Aim}
     
At this step, we have the list of inequalities as in
Theorem~\ref{th:Kron} excepted that the condition {\it $\pi$ birational} has
been replaced by the weaker one {\it $\pi$ dominant}. 
It is well known (see e.g. \cite{VW:ineq}) that the current list of inequalities
characterizes the cone $\Cc(\VV)$ but contains redundant
inequalities. For example, at this step, the Kronecker cone $\Cc(\VV)$
for $(4,4,4)$ is characterized by 1260 inequalities (230 up to
$\mathfrak S_3$-symmetry plus 3 dominancy inequalities) while the minimal list has 270 inequalities (47  up to
$\mathfrak S_3$-symmetry plus dominancy inequalities).

Software of convex geometry like
Normaliz~\cite{normaliz} allows you to select the irredundant
inequalities. However, our attempts to use it for Kron(5,5,5) proved unsuccessful due to insufficient memory.
We propose another algorithm based on Lemma~\ref{lem:RamContract}.
In the manageable cases, our algorithm is also faster: for Kron(4,4,4), from the 1260 inequalities from Step 4,
it takes 2 min 19 s to Normaliz to extract the 270 irredundant ones. It took less than 4 seconds for our algorithm to extract these inequalities. 
Note that even if Normaliz might also provide additional information like extreme rays of the cone, 
it works faster when fed directly with the list of irredundant inequalities: less than 1s for Kron(4,4,4). 

\bigskip
So, we fix a pair $(\tau,w)$ such that $\pi$ is dominant and
Condition~\ref{ass:Kron2} holds. Our Step 5 allows us to decide whether
$\pi$ is birational or not.

\subsection{Mathematical background}

\subsubsection{Birational maps and their ramification divisors}
\label{subsec:ramif_div}

Let $\pi\,:\,X\longto Y$ be a dominant morphism between two irreducible
varieties $X$ and $Y$ with common dimension $N$.
Assume that $X$ is smooth in codimension one.
Let $X^{s}$ denote its smooth locus.
The determinant $J$ of $T\pi$ is a section of the determinant line bundle on $X^s$
$$
\Dc:=\bigwedge^NT^*X^s\otimes \pi^*(\bigwedge^NTY).
$$
Its zero locus $\div(J)$ is a Cartier divisor on $X^s$. The closure
$R_\pi$ of $\div(J)$ is a well-defined Weyl divisor on $X$, called the
{\it ramification divisor of $\pi$}. 
Let $\supp(R_\pi)$ denote its support: it is a closed reduced subset
of $X$. 

A divisor $D$ on $X$ is said to be {\it contracted (by $\pi$)} if the closure  
$\overline{\pi(\supp(D))}$ has codimension at least two in $Y$. 
    
    \begin{lemma}
      \label{lem:RamContract}
      Let $\pi\,:\,X\longto Y$ be a dominant morphism with $X$ and $Y$ of the same dimension. 
      Assume also that $X$ is smooth in codimension one.
      Then, the following two assertions are equivalent:
      \begin{enumerate}[label=(\roman*), ref=(\roman*)]
      \item \label{ass:rc1} $R_\pi$ is contracted;
        \item \label{ass:rc2} For any $x\in
          \supp(R_\pi)$ the kernel of $T_x\pi$ intersects nontrivially $T_x\supp(R_\pi)$.
      \end{enumerate}
If, moreover $\pi$ is proper, $Y$ is smooth and simply connected then the
      previous assertions are also equivalent to:
      \begin{enumerate}[label=(\roman*), ref=(\roman*), start=3]
     % \item \label{ass:rc1bis} $R_\pi$ is contracted; 
          \item \label{ass:rc3} $\pi$ is birational.
      \end{enumerate}
    \end{lemma}
    
    \begin{proof}
The condition of assertion~\ref{ass:rc2} is closed in $x$ so it can equivalently be checked for $x$ general 
in each irreducible components $Z$ of $\supp(R_\pi)$. 
On such a point, the condition means that $T_x(\pi_{|Z})$ is not injective. 
By Lemma~\ref{lem:domTpi} applied to $\pi_{|Z}: Z\rightarrow \pi(Z)$, this amounts to $\dim \pi(Z)<\dim Z$. 
Hence, assertions \ref{ass:rc1} and \ref{ass:rc2} are equivalent.

Assuming \ref{ass:rc3}, the ramification divisor is contracted by main's Zariski (see \cite[Chap~III Section~9 Proposition~1]{Mumford:red}).

Conversely, we reproduce the argument of \cite[Proposition~12]{FR:BKGB} in our slightly different context.  
Using the Stein factorization \cite[\href{https://stacks.math.columbia.edu/tag/03GX}{Tag 03GX}]{stacks-project}, we may
  assume that $\pi$ is finite. 
  Since $\pi$ is proper and dominant, it is surjective.
  Let $Z$ denote the singular locus of $X$.
  Then $\pi$ is a covering from $X^s\backslash
  \supp(R_\pi)$ onto $Y\backslash \pi(\supp(R_\pi)\cup Z)$.
 Then, since  
 $\pi(\supp(R_\pi)\cup Z)$ has codimension at least two in $Y$, the
  fundamental groups of $Y$ and $Y\backslash \pi(\supp(R_\pi)\cup Z)$ coincide,
  and $Y\backslash \pi(\supp(R_\pi)\cup Z)$ is simply connected.
  Thus, $\pi$ has degree one and the assertion is proved.
    \end{proof}

\subsubsection{Checking contraction of divisors}

Let $E$ and $V$ be two vector spaces of the same dimension $N$ and
let $\pi\,:\,E\longto V$ be a dominant morphism. 
We need an effective criterion to decide if the ramification divisor $R_\pi$ is
contracted.

Fix bases of $E$ and $V$, and let $M(x)$ be the matrix of $T_x\pi$ in
these bases. Then $M$ is a matrix with entries in $\CC[E]$. 

Let $J\,:\,E\longto\CC,x\longmapsto \det(T_x\pi)$ be the Jacobian determinant of $\pi$. 
It is an element of the UFD ring $\CC[E]$. Write its decomposition as
product of irreducible polynomials
$$
J=\prod_{i=1}^t f_i^{m_i},
$$
and let 
$$
J_{red}=\prod_{i=1}^t f_i,
$$
be the reduced equation of $\supp(R_\pi)$.

Choose $(a,b) \in\Aff_1:=(E\setminus\{0\})\times E$, and consider the affine-linear map $\varphi\,:\,\CC\longto E,
z\longmapsto az+b$ and its image $\ell$.
Set $M_z=M\circ\varphi$. It is a matrix with coefficients in the
principal ring $\CC[z]$. By Smith's theorem, there exists $U,V\in\GL_N(\CC[z])$ and a diagonal
matrix such that
\begin{equation}
UM_z V=\diag(s_1,\dots,s_N)\label{eq:Smith}
\end{equation}
with each $s_i$ dividing $s_{i+1}$. 
In addition, $\sum_i\deg(s_i)=\deg (\det M_z)\leq N$. 
Set 
\begin{equation}
  \label{eq:defdelta}
\delta_1=\frac{s_N}{\gcd(s_{N-1}^N, s_N)}
\quad\mbox{ and }\quad
\delta=\frac{\delta_1}{\gcd(\delta_1,\delta_1')}.
\end{equation}
For general $(a,b)$, $\det M_z$ is nonzero so $s_N$ and $\delta$ are nonzero either.

Let $V_N$ be the last column of $V$.

\begin{theo}
\label{th:ram0}
With above notation, assume that $(a,b)\in \Aff_1$ is general. 
Then, the divisor  $R_\pi$ is contracted if and only if $\delta$ divides 
 $T_{\varphi(z)} J_{red}(V_N(z))$ in $\CC[z]$. 
\end{theo}

\begin{proof}
Rewrite the divisor $R_\pi$ as 
$$
R_\pi=\sum_{i\in I} m_iD_i+\sum_{k\in K} n_k E_k
$$
where $m_i$, $n_k$ are positive integers, and the $E_i$ and $D_j$ are
pairwise distinct prime divisors such that
$$
\forall i\in I\quad\exists x\in D_i^s\quad \rk(T_x\pi)=N-1,
$$
and
$$
\forall k\in K\quad\forall x\in E_k\quad \rk(T_x\pi)\leq N-2.
$$

Let $(f_i)_{i\in I}$ and $(g_k)_{k\in K}$ be equations for the $D_i$'s
and $E_k$ respectively. 
Let $d_i$ and $e_k$ denote the degrees of $f_i$ and $g_k$ respectively.
Note that
$$
J=\prod_{i\in I} f_i^{m_i}\times \prod_{k\in K} g_k^{n_k},
$$
%and
$$
J_{red}=\prod_{i\in I} f_i\times \prod_{k\in K} g_k,
$$
up to scalars that we modify to get the formulas.
We distinguish between the components $D_i$ and $E_k$, because of the
following consequence of Lemma~\ref{lem:domTpi}.

\smallskip

\noindent\underline{Claim 1.} Each divisor $E_k$ is contracted by
$\pi$.

\smallskip
Claim 1 implies that  
\begin{equation}
  \label{eq:5}
R_\pi \mbox{ is contracted  }\; \Leftrightarrow \;
  \forall i\in I,
D_i\mbox{ is contracted by }\pi.
\end{equation}

Write now the factorizations of $s_{N-1}$ and $s_N$:
$$
s_{N-1}=\prod_i (z-\xi_i)^{\alpha_i},\qquad
s_N=\prod_i (z-\xi_i)^{\beta_i}\times \prod_j (z-\zeta_j)^{\gamma_j},
$$
where the $\xi_i$'s and $\zeta_j$'s are pairwise distinct, the
$\alpha_i$'s, $\beta_i$'s and $\gamma_j$'s are positive integers such
that $N\geq \beta_i\geq \alpha_i$.
Observe that
$$
\delta=\prod_j (z-\zeta_j).
$$
By definition of the $D_i$ and of $s_N$, it is clear that $\ell\cap\bigcup_i D_i=\{\varphi(\zeta_j)\}$.
By genericity of $(a,b)$, $\ell\cap\bigcup_i D_i$ contains at least a point (more precisely, exactly $d_i$ points) 
on each irreducible component $D_i$. 
Moreover, we assume that each $\zeta_j$ lies in a single component $D_{i(j)}$ and is sufficiently general in it. 
Then, $\sum_i D_i$
is contracted by $\pi$ if and only if
\begin{equation}
  \label{eq:7}
  \forall j\qquad
T_{\varphi(\zeta_j)}D_{i(j)}\cap \Ker  M_z(\zeta_j)\neq\{0\}.
\end{equation}
On the one hand, for any such $\zeta_j$, \eqref{eq:Smith} implies that $\Ker
 M_z(\zeta_j)$ is spanned by $V_N(\zeta_j)$. 
On the other hand, $T_{\varphi(\zeta_j)}D_{i(j)}$ is the Kernel of $T_{\varphi(\zeta_j)} J_{red}$. 
We get that $R_\pi$ is contracted if and only if 
\begin{equation}
  \label{eq:9}
  \forall j\qquad T_{\varphi(\zeta_j)} J_{red} (V_N(\zeta_j))=0.
\end{equation}
The $\zeta_j$'s being the simple roots of the polynomial $\delta$,
this is equivalent to 
\[
\delta\mbox{ divides } T_{\varphi(z)} J_{red} (V_N(z)) \mbox{ in } \CC[z]. 
\] 
\end{proof}

Let $x_1,\dots,x_N$ be the variables on $E$.
Theorem~\ref{th:ram0} gives an algorithm to decide if $R_\pi$ is contracted or not, 
since every data $ J_{red}$, $V_N$ and $\delta$ can be computed explicitly.
Start by choosing randomly a pair $(a,b)\in \Aff_1$. 

\begin{enumerate}
  \item $J$ is the determinant of an explicit matrix $M=(\frac{\partial \pi_i}{\partial x_j})$.
  \item The polynomial $ J_{red}$ can be computed by factorizing $J$. 
  \item Smith's algorithm (see e.g. \cite{Villard:Smith}) allows you to compute, starting with $ M_z$, 
  the matrices $U$ and $V$ and the polynomials $s_1,\dots,s_N$. 
  We deduce $\delta$ and $V_N$ easily. 
\end{enumerate}

\begin{remark}
  Taking $(a,b)$ the generic point of $\Aff_1$, and hence working in
  $\CC(\Aff_1)[z]$, one gets a deterministic algorithm to decide if
  the divisor is contracted. Unfortunately, it is not efficient.
  It works only up to Kronecker $(3,2,2)$.
\end{remark}
\smallskip

In Section~\ref{sec:raf5}, we give alternative methods to compute 
these data that speed up considerably the algorithm. 

\subsubsection{Compactification of the map  \texorpdfstring{$\pi$}{pi} of Theorem~\ref{th:Kron}}
\label{sec:bircompact}

We cannot apply Lemma~\ref{lem:RamContract} directly to the map $\pi$ of
Theorem~\ref{th:Kron}  because it is not proper. 
This is why we introduce a partial compactification of its domain.

Consider the parabolic subgroup $P(-\tau)$.
Let $Y$ be a  $P(-\tau)$-stable locally closed subvariety of some $G$-variety $X$. 
The quotient of $G\times Y$ by  the following action of $P(-\tau)$:
$$
p.(g,v):=(gp\inv,pv)
$$
identifies with
$$
\Yc=\{(gP(-\tau)/P(-\tau),x)\in G/P(-\tau)\times X\;:\;g\inv x\in Y\}
$$
by the map
$$
(g,v)\longmapsto (gP(-\tau)/P(-\tau),gv).
$$
Viewed as the quotient, $\Yc$ is denoted by $G\times_{P(-\tau)}Y$.
Observe that, in  this construction of $G\times_{P(-\tau)}Y$ one can
replace $G$, by any locally closed subvariety $Z$ of $G$ that is stable by
right multiplication by $P(-\tau)$.

We now fix a pair $(\tau,w)$ such that $\pi$ is dominant and
Condition~\ref{ass:Kron2}holds.
Observe that $\pi$ identifies with 
$$
\pi\,:\,(w\inv B^-wP(-\tau))\times_{P(-\tau)}\VV^{\tau\leq 0}\longto\VV.
$$
Indeed, since $w\in W^{P(\tau)}$, the map $(w\inv U^-w\cap U)\longto
w\inv B^-wP(-\tau)/P(-\tau)$, $u\longmapsto uP(-\tau)/P(-\tau)$ is an isomorphism.
Consider the following subvarieties of $G/P(-\tau)$:
$$
\begin{array}{ll}
  X_w=w\inv \overline{B^-wP(-\tau)/P(-\tau)}&\mbox{the translated
                                                 Schubert variety;}\\[1em]
  X_w^\circ=w\inv B^- wP(-\tau)/P(-\tau)&\mbox{the translated
                                                 Schubert cell.}
\end{array}
$$
Consider
$$
\begin{array}{cccl}
\bar\pi\,:&\Xc:=\tilde X_w\times_{
              P(-\tau)}\VV^{\tau\leq 0}&\longto&\VV\\
            &[g:v]&\longmapsto&gv,
\end{array}
$$
where $\tilde X_w$ denote the pullback
in
$G$ of $X_w\subset G/P(-\tau)$ by $G\longto G/P(-\tau)$.
We can see $G\times_{P(-\tau)} V^{\tau\leq0} \rightarrow G/P(-\tau)$ as a locally trivial bundle with fiber isomorphic to $V^{\tau\leq 0}$. So $\Xc$ is locally isomorphic to $X_w\times V^{\tau\leq0}$. Since $X_w$ is normal (see e.g. \cite{BrionKumar}), $\Xc$ is also normal.

Observe that
$$
\begin{array}{cccl}
\Xc&\longto&G/P(-\tau)\times \VV\\
(g,v)&\longmapsto&(g P(-\tau)/P(-\tau) ,gv)
\end{array}
$$
is an immersion.
In particular, we have:
\begin{lemma}
\label{lem:piproper}
  The map $\bar\pi$ is proper.
\end{lemma}

\smallskip\noindent{\bf The boundary.}
Let $\Delta$ be the set of simple roots. On $W$, we denote by $\leq$
the Bruhat order.
Consider now the boundary:
\begin{equation}
  \label{eq:bordXw}
  \partial X_w=X_w-X_w^\circ=\bigcup_{\begin{subarray}{c} v\in W^{P(\tau)},\, v\leq  w,\\ 
\ell(v)=\ell(w)-1 \end{subarray}}
w\inv Y_v,
\end{equation}
where $Y_v=\overline{B^-vP(-\tau)/P(-\tau)}=vX_v$.

  By the properties of the Bruhat order, for any $v$ appearing in
  \eqref{eq:bordXw}, 
  $v=s_\beta w$ for some positive root $\beta$.
  We distinguish two cases:
  \begin{enumerate}
  \item weak Bruhat order: $\beta\in\Delta$;
    \item strictly strong Bruhat order: $\beta\not\in\Delta$.
    \end{enumerate}
    The point is that this dichotomy reflects the relative position of
    the corresponding component of the boundary and the
    ramification divisor of $\bar\pi$.

For $v=s_\beta w$ as in \eqref{eq:bordXw}, set
$$
D_\beta=w\inv \tilde Y_v\times_{P(-\tau)}\VV^{\tau\leq 0}.
$$

Since $\Xc$ is normal, it is smooth in codimension 1 and we consider the ramification divisor $\Rc$ of $\bar \pi$ as defined in Section~\ref{subsec:ramif_div}.

\begin{lemma}
\label{lem:posalphaR}
  Let $\alpha\in\Delta$ such that $\ell(v)=\ell(w)-1$ where
  $v=s_\alpha w$.

Then $D_\alpha$ is not contained in the support of $\Rc$.
\end{lemma}

\begin{proof}
  Consider the minimal parabolic $P^-(\alpha)$ containing $B^-$
  associated to $\alpha$. Then, $X_w$ is $w\inv P^-(\alpha)w$-stable
  and
  $\bar\pi\,:\overline{w\inv
    P^-(\alpha)wP(-\tau)}\times_{P(-\tau)} \VV^{\tau\leq 0}\longto \VV$ is
  $w\inv P^-(\alpha)w$-equivariant.

  Since $w\inv P^-(\alpha)w D_\alpha$ is the open subset $w\inv
    P^-(\alpha)wP(-\tau)\times_{P(-\tau)}  \VV^{\tau\leq 0}$, one deduces that
    $D_\alpha$ cannot be contained in the $w\inv P^-(\alpha)w$-stable subset $\supp(\Rc)$ of codimension $1$.
\end{proof}

\begin{lemma}
\label{lem:posbetaR}
  Let $\beta$ be a positive root that is not a simple root such that $\ell(v)=\ell(w)-1$ where
  $v=s_\beta w$.

Then $D_\beta$ is contained in the support of $\Rc$.
\end{lemma}

\begin{proof}
  Consider the restriction $\pi_\beta$ of $\pi$ to $D_\beta$. 
Let $v\in V^{\tau \leq 0}$. 
One easily checks that the matrix of the tangent map of $T_{[e:v]}\pi_\beta$ is block triangular. 
More like in \cite[Theorem~7.4]{BKR}, one can prove that the shape of these blocks implies that  $T_{[e:v]}\pi_\beta$ 
cannot be injective. In \emph{loc. cit.}, the hypothesis on the $Q_w$-orbit amounts to assume $\beta$ is not simple.
\end{proof}

\subsection{Application to our algorithm}

Fix again a pair $(\tau,w)$ such that $\pi$ is dominant and
Condition~\ref{ass:Kron2} holds. 
  
Consider the map $\bar\pi$ constructed in the previous section and its ramification divisor $\Rc$. Note that all the assumptions of Lemma~\ref{lem:RamContract} are satisfied for $\bar \pi$ and this map is birational if and only if $\pi$ is. Thus
\begin{equation}
\pi \mbox{ is birational if and only if }  \Rc\mbox{ is contracted.}
\end{equation}

Consider the inclusion $\iota$ of the domain of $\pi$ in that of $\bar\pi$. 
We have $\Rc_\pi=\iota^*(\Rc)$. Then, by Lemmas~\ref{lem:posalphaR} and \ref{lem:posbetaR}, we have
\begin{equation}
  \supp(\Rc)=\overline{\supp(\Rc_\pi)}\cup\bigcup_{
    \begin{subarray}{c}
      \beta\in\Phi^+-\Delta \; { \rm s.t.}\\
      s_\beta\in W^{P(\tau)}, \;\ell(s_\beta w)=\ell(w)-1
    \end{subarray}
  } D_\beta.
\end{equation}

\smallskip
We have $D_{\beta}\cong v^{-1}wD_{\beta}= v\inv \tilde Y_v\times_{P(-\tau)}\VV^{\tau\leq 0}$. 
Similarly to Section \ref{sec:bircompact}, on the open cell $D_{\beta}^{\circ}$ of $D_{\beta}$,
 the map $\bar\pi_{|D_{\beta}^{\circ}}$ identifies with the (nondominant) map
$$
\pi_v\,:\,(v\inv U^-v\cap U)\times\VV^{\tau\leq 0}\longto\VV.
$$
In particular $D_\beta$ is not contracted by $\bar\pi$ if and only if the tangent map of $\pi_v$
at $(e,x)$ for $x$ general in $\VV^{\tau\leq 0}$ is injective.
This last condition is very similar to the dominancy condition of $\pi$ and can be checked following the same process, see Section~\ref{sec:Step4}.   

\smallskip
In the Kronecker case for $d=(4,4,4)$, over the 183 redundant inequalities considered at this step
(up to $\mathfrak S_3$-symmetry), 64  are eliminated  by the detection of some noncontracted $D_\beta$.
        
\smallskip
To decide whether $\Rc_\pi$ is contracted or not, we use Theorem~\ref{th:ram0} as explained after its proof.

\section{The Kronecker, Fermionic and Bosonic cones}
\label{sec:3cones}

\subsection{Representations of \texorpdfstring{$\GL_n(\CC)$}{GLnC}}

Fix a positive integer $n$ and consider the polynomial representations of
the linear group $\GL_n(\CC)=\GL(V)$, where $V$ is a fixed
$n$-dimensional complex vector space.  
Here polynomial means that we do not use $\det\inv$. 
The irreducible polynomial representations are parameterized by the partitions
$\lambda$ of length $\ell(\lambda)$ at most $n$. The representation associated to $\lambda$ is called the
Weyl module of highest weight $\lambda$ and is denoted by $S^\lambda
V$($=V_{GL_n(\CC)}(\lambda)$ in the notation of \S \ref{sec:intro_cone}).
For example $S^{1^r}V=\bigwedge^rV$ and $S^rV$ is the symmetric power
of $V$.

\subsection{Some multiplicities for branching problems}

Consider the complex representations of
the symmetric group $\mathfrak S_n$. 
The irreducible representations are parametrized by the partitions
$\lambda$ of $n$. The representation associated to $\lambda$ is called the
Specht module of type $\lambda$ and is denoted by $[\lambda]$.

Consider now, the problem of decomposing the tensor products of Specht
modules. Hence, given $s$ partitions $\lambda_1,\dots,\lambda_s$ of $n$,  denote by
$$
g_{\lambda_1,\lambda_2,\dots,\lambda_s}=\dim\bigg(
([\lambda_1]\otimes [\lambda_2]\otimes\cdots\otimes[\lambda_s])^{\mathfrak S_n}
\bigg)
$$
the multiple Kronecker coefficient, which coincides with the multiplicity 
of $[\lambda_s]$ in $[\lambda_1]\otimes [\lambda_2]\otimes\cdots\otimes[\lambda_{s-1}]$.
Observe that, by associativity of the tensor product, we have, for any $s\geq 4$
\begin{equation}
  \label{eq:assKron}
  g_{\lambda_1\dots,\lambda_s}=\sum_\mu g_{\lambda_1,\dots,\lambda_{s-2},\mu}\;.\;g_{\mu,\lambda_{s-1},\lambda_s}.
\end{equation}
Fix $s$ finite dimensional complex vector spaces $V_1,\dots,V_s$.
By Schur-Weyl duality, one has, under the action of $G=\GL(V_1)\times\cdots\times\GL(V_s)$
\begin{equation}
  \label{eq:SWdualmK}
  S^n (V_1\otimes\cdots\otimes
  V_s)=\bigoplus_{(\lambda_1,\dots,\lambda_s)}g_{\lambda_1,\dots,\lambda_s}
  S^{\lambda_1}V_1\otimes\cdots\otimes S^{\lambda_s}V_s,
\end{equation}
where the sum runs over the $s$-uples of partitions of $n$ such
that $\ell(\lambda_i)\leq \dim(V_i)$.

\bigskip

Similarly, given $s+1$ partitions $\lambda_1,\dots,\lambda_s,\nu$ of
length at most $n$,  denote by
$$
c_{\lambda_1,\dots,\lambda_s}^\nu=\dim\bigg(
(S^\nu V^*\otimes S^{\lambda_1}V\otimes\cdots\otimes S^{\lambda_s}V)^{\GL(V)}
\bigg)
$$
the multiple Littlewood-Richardson coefficient. 
Observe that, by associativity of the tensor product, we have, for any $s\geq 3$
\begin{equation}
  \label{eq:assLR}
  c_{\lambda_1\dots,\lambda_s}^\nu=\sum_\mu
  c_{\lambda_1,\dots,\lambda_{s-2},\mu}^\nu \;\cdot\;c_{\lambda_{s-1},\lambda_s}^\mu.
\end{equation}

\bigskip
Given a vector space $V$ and two partitions $\mu$ and $\theta$ such
that $\ell(\theta)\leq \dim(V)$ and $\ell(\mu)\leq \dim(S^\theta V)$,
consider the irreducible $\GL(S^\theta V)$-representation
$S^\mu(S^\theta V)$. Under the $\GL(V)$-action, it decomposes as
$$
S^\mu(S^\theta V)=\bigoplus_{\lambda}a(\mu,\theta,\lambda) S^\lambda V,
$$
where the sum runs over the partitions $\lambda$ such that
$\ell(\lambda)\leq\dim(V)$ and $|\lambda|=|\mu|.|\theta|$.
The multiplicities $a(\mu,\theta,\lambda)$ are called {\it Plethysm
  coefficients}.

\subsection{Three moment cones}
\label{subsec:3mc}
Fix $s$ positive
integers $\ud=(d_1\geq d_2\geq\cdots\geq d_s)$  and $s$ vector spaces  $V_k$ of respective dimension $d_k$.
Let $G_{\ud}:=\GL(V_1)\times\cdots\times\GL(V_s)$. When $s=1$, we may simply replace notation $\ud$ and $d_1$
 by $d$ and we may omit any index or superscript $k$ referring to $V_k$.

We also fix a basis on each vector space $V_k$ thus allowing us to identify endomorphisms and matrices.
Then $T=T_{\ud}$ (resp. $B$) denotes the product of the groups of diagonal (resp. upper triangular) invertible matrices.
Hence, elements of $T$ are $s$-uples of diagonal matrices.
For any $1\leq k\leq s$ and $i\leq d_k$, let $\epsilon^k_i\in
X(T)$ sending an element of $T$ to the $i^{th}$ diagonal element of
its $k^{th}$ matrix.

We also fix the Hermitian product on each $V_k$ making the fixed
basis unitary. Then $K=U(V_1)\times\cdots\times U(V_s)$ is a maximal
compact subgroup of $G_{\ud}$.

{\bf Kronecker case.}
Consider  $\VV=\VV_{\ud}=V_1\otimes V_2\otimes\cdots\otimes V_s$ endowed with 
its natural action of
$G_{\ud}$.

Then $\CC[\VV]=\oplus_{l\geq 0}S^l\VV^*$ is a $G_{\ud}$-module that decomposes as (see \eqref{eq:SWdualmK}):
\begin{equation}
\bigoplus_{l\geq 0}\bigoplus_{
  \begin{subarray}{c}
    (\lambda_1,\dots,\lambda_s)\; \; s \; { \rm  partitions}\\
    \ell(\lambda_i)\leq d_i\\
    |\lambda_1|=\cdots=|\lambda_s|=l
  \end{subarray}}
g_{\lambda_1,\dots,\lambda_s}S^{\lambda_1}V_1^* \otimes
\cdots \otimes S^{\lambda_s}V_s^*.\label{eq:decCV}
\end{equation}

Hence,
$$
S(\VV_{\ud},G_{\ud})=\left\{\ulambda^*=(\lambda_1^*,\dots,\lambda_s^*) \,:\,
\begin{array}{l}
\lambda_k \mbox{
  are partitions such that}\\
\mbox{$\ell(\lambda_k)\leq\dim(V_k)$, and $g_{\lambda_1,\dots,\lambda_s}\neq 0$}
\end{array}
\right\} 
$$
is, modulo the twist $\lambda\mapsto\lambda^*$, the multiple Kronecker semigroup.

\bigskip
The set of weights of $T$ on $\VV$ is
$$
\Wt(\VV):=\{[a_1,a_2,\dots,a_s]:=\sum_{k=1}^s\epsilon^k_{a_k}\,:\,\forall
 k\in [\![1,s]\!], 
1\leq a_k\leq d_k\}.
$$

The cone $\Cc(\VV_{\ud},G_{\ud})$ is contained in the subspace defined by the condition
$|\lambda_1|=\cdots=|\lambda_s|$.  
This is explained by the fact that the $(s-1)$-dimensional torus
$$
Z=\{(t_1\Id_{V_1}, t_2\Id_{V_2},\dots,
t_s\Id_{V_s})\,:\,t_i\in\CC^* \mbox{ s.t. }\prod_{k=1}^st_k=1\}\subset T
$$
acts trivially on $\VV_{\ud}$.
Then the group of characters $X(T/Z)$ is the subgroup of $X(T)$ characterized by the conditions 
$|\lambda_1|=\cdots=|\lambda_s|$. 
In order to apply Theorem \ref{th:Kron}, we simply replace $G_{\ud}$ by $G_{\ud}/Z$.

For $1\leq k\leq s$ and $1\leq i\leq d_k$, let $\epsilon^{k,*}_i\in
X_*(T)$, sending $t\in \C$ on the tuple of diagonal matrices having $t$ at
the $i^{th}$ entry of the $k^{th}$ matrix and $1's$ elsewhere.
They form a basis $\Ec$ of $X_*(T)$. Letting $\Ec_0:=\{\epsilon^{k,*}_{d_k}\,:\, k<s\}$, 
the family $\Ec\setminus\Ec_0$ yields a basis of $X_*(T/Z)=X_*(T)/X_*(Z)$.
We choose to identify $X_*(T/Z)$ to the subgroup of $X_*(T)$ imposing null coordinates in the elements of $\Ec_0$. An element $\tau$ belonging to this subgroup is said to be {\it normalized}, see later example \ref{ex:taubar}.

\bigskip
{\bf Fermionic and Bosonic cases.}
Here $s=1$, so $V$ is a $d$-dimensional vector space.
Fix also a positive integer $r$. In the fermionic case, assume that
$r<d$. Set
$\theta=1^r$ in the fermionic case and $\theta=r$ in the bosonic case.
Consider the $G_d$-representation $\VV=S^\theta V$.  
Then
$$
S(\VV,G_d)=\left\{\lambda^* \,:\,
\begin{array}{l}
\lambda \mbox{
  is a  partition such that $\ell(\lambda)\leq\dim(V)$, and }\\
a(d,\theta,\lambda)\neq 0, \mbox{ where }rd=|\lambda|
\end{array}
\right\}. 
$$
In the fermionic case, the set of weights of $T$ on $\VV$ is
$$
\Wt(\VV):=\{[i_1<i_2<\dots<i_r]:=\sum_{t=1}^r\epsilon_{i_t}\,:\,\forall
1\leq i_t\leq d
\}.
$$
In the bosonic case, the set of weights of $T$ on $\VV$ is
$$
\Wt(\VV):=\{[i_1\leq i_2\leq \dots\leq i_r]:=\sum_{t=1}^r\epsilon_{i_t}\,:\,\forall
1\leq i_t\leq d
\}.
$$

\section{Some refinements in Step 1}
\label{sec:raf1}

\subsection{The general case}
\label{sec:step1_raf_gen}

Let us present two improvements of Algorithm~\ref{algo:Sreg}. 

First, one can anticipate Step 2 by keeping only the $\tau$'s such that
\begin{equation}  
  \label{eq:2veryweak}
  \sum_{\begin{subarray}{c}
    \chi\in \Wt(\VV)\\
    \langle\tau,\chi\rangle>0 
  \end{subarray}
  }m_\chi\leq
    \dim(U(\tau)). \tag{B''}
    \end{equation}
Indeed, condition \eqref{eq:2veryweak} is a consequence of $\eqref{eq:2weak}$. 
This observation gives an exit condition in the recursive Algorithm~\ref{algo:Sreg}: 
if at some step the cardinality of $S_+$ exceed $\dim(U(\tau))$ one can stop the inspection.
Observe that, $U(\tau)$ only depends on the face of $X_*(T)$ containing $\tau$ in its interior. 
Thus, in our algorithm, it is known before $\tau$ is. 

For example, in the case of Kron(5,5,5), when we are looking for regular $\tau$, 38 weights over the 125 ones 
are necessarily in $S_-$ by this condition.  
These 38 weights are never considered as elements in a candidate $S_0$ by our algorithm.

\bigskip
At the line \ref{line:choose_chi}, Algorithm~\ref{algo:Sreg} chooses an indeterminate weight $\chi$.
Any choice makes the algorithm correct, but the execution time depends on it. 
Basically, we choose $\chi$ to maximize the number of indeterminate $\chi'$ comparable to $\chi$.
Indeed, these weights will be affected to $S_\pm$ in the branch where $\chi$ is put in $S_0$. 

At line~\ref{line:Cp} of Algorithm~\ref{algo:Sreg}, we compute the rank of a matrix whose the rows are characters.
Since we compute a lot of ranks for matrices having the same first rows.
We may want to compute echelon forms of matrices as we add weights to $S_0$.
In our Python-Sage implementation this method is not faster, but it could be in other languages. 

\subsection{The Kronecker case}

Now, we look at some features specific to the Kronecker case as in Section~\ref{subsec:3mc}. 

\bigskip
\noindent{\bf Using symmetries.}
To simplify the exposition, assume here, that $\ud=(d,d,d)$ for some $d>1$. 
Then our problem and, in particular, the cone $\Cc(\VV_{\ud},G_{\ud})$ is stable under the $\mathfrak S_3$-symmetry 
corresponding to permutations on the three factors.

In Step~1, it is sufficient to describe the $\tau$'s up to $\mathfrak S_3$-symmetry.
A way to exploit this in Step~1 runs as follows. 
Assume that, in one recursive instance in Algorithm~\ref{algo:Sreg}, $S_{\neq 0}$ is nonempty while $S_0$ is empty. 
Then one can assign in $S_{\neq 0}$ the $\mathfrak S_3$-orbit of each one of its elements. 
This allows us to affect the characters to $S_{\neq 0}$ more quickly and, hence, speed up the algorithm. 

This observation is refined in the implemented version to deal with the symmetry groups of each face of $X_*^+(T)$. 

\bigskip
\noindent{\bf About the faces of the set of dominant one-parameter subgroups.}

Recall from \S\ref{subsec:3mc} that we have to replace the group $G$ by its quotient $G/Z$ in this case, 
that $X(T/Z)$ is a subset of $X(T)$ and that $X_*(T/Z)$ is identified to a subset of $X_*(T/Z)$.

Condition~\eqref{eq:3weak} becomes:
\begin{equation}\label{eq:3weakkron}
\Wt_T(\VV)^{\tau} \mbox{spans an $s$-codimensional subspace of } X(T)_{\QQ}.
\tag{C'${}_{Kron}$}
\end{equation}

\begin{remark}\label{rk:finiteness}
Note that if an indivisible dominant normalized $\tau$ satisfies condition \eqref{eq:3weak}, then $\Q\tau+X_*(Z)_{\QQ}$ is uniquely determined as the orthogonal of $\Wt(\VV)^\tau$. So we can recover $\Q\tau$ from its set of orthogonal weights in $\Wt(\VV)$ by the normalization hypothesis, and thus $\tau$ by indivisibility and dominance.

%In particular, there are at most $\begin{pmatrix}\sharp\Wt(\VV)\\ \dim T-s \end{pmatrix}=\begin{pmatrix}\prod_{k=1}^s d_k\\ \sum_{k=1}^s (d_k-1)\end{pmatrix}$ normalized $\tau$ satisfying condition \ref{ass:Kron3}.
\end{remark}

Let $\tau$ be a dominant one-parameter subgroup of $T=T_{\ud}$. 
Let $\Fc$ be the face of $X_{*}^+(T)$ containing $\tau$ in its interior. 
For $k\in [\![1,s]\!]$, we define $\bar d_k=\dim T_{\Fc} \cap \GL(V_k)$ so that $T_{\Fc}$ identifies 
with $T_{\bar \ud}$. Under this identification, the regular element $\bar\tau\in X_*(T_{\bar \ud})$ corresponding 
to $\tau\in X_*(T_{\ud})$ is obtained by merging repetitions of coefficients in the $\epsilon^{k,*}_i$. 

\begin{exple}\label{ex:taubar}
  If $d=(5,5,5,1)$ and
  $$\begin{array}{ll}
  \tau&=(3\,2\,2\,0\,0\,|\,2\,2\,0\,0\,0\,|\,3\,-1\,-1\,-1\,0\,|\,5)\\
  &= 3\epsilon^{1,*}_1+2\epsilon^{1,*}_2+2\epsilon^{1,*}_3+2\epsilon^{2,*}_1+2\epsilon^{2,*}_2\dots
  \end{array}
  $$

  Then $\bar\tau=(3\,2\,0\,|\,2\,0\,|\,3\,-1\,0\,|\,5)$ and
  $\bar{\ud}=(3,2,3,1)$.
\end{exple}

The inclusion $T_{\bar\ud}\rightarrow T_{\ud}$ yields a projection $p_\Fc: X(T_{\ud})\rightarrow X(T_{\bar \ud})$ 
identifying some of the $\epsilon^k_i$.
Our main observation it that $p_\Fc(\Wt_T(\VV))$ only depends on $\bar\ud$ (and not on $\Fc$), and even more identifies 
with $\Wt(\VV_{\bar \ud})$.

In Section~\ref{sec:Step1}, Algorithm~\ref{algo:Sreg} has to be applied to any face $\Fc$ of $X_*^+(T)$. 
The observation allows us to replace this loop on the faces by a loop on the possible $\bar\ud$. 
For example, for Kron$(5,5,5)$, up to $\mathfrak S_3$-symmetry, it is sufficient to run over the $\bar\ud$
that are partitions contained in 5\,5\,5 and with no zero part. 
So, to apply Algorithm~\ref{algo:Sreg}, we run over 35 partitions, instead of running over the $816$ faces up to symmetries. 

Moreover, concerning Condition \eqref{eq:3weakkron}, the routine of Algorithm \ref{algo:Sreg} can be used directly for $\bar\ud$ thanks to the following Lemma. Note that we only need a necessary condition here since the goal of Step 1 is only to produce a finite list of candidates that is refined in the other steps.

\begin{lemma}\label{lem:chitoreg}
  If $\Wt(\VV_{\ud})^{\tau}$ spans an
  $s$-codimensional subspace in $\Q X(T_{\ud})$ then $\Wt(\VV_{\bar \ud})^{\bar \tau}$ spans an
$s$-codimensional subspace in $\Q X(T_{\bar \ud})$.
\end{lemma}

\begin{proof}
 Note that, by construction, $\langle \tau, \chi\rangle=\langle \bar \tau,p(\chi)\rangle$. 
 Thus, the latter set is the projection of the former one and the codimension of its span cannot exceed $s$. Since the dimension $s$ subspace $\Q\bar \tau+X_*(Z_{\bar\ud})_{\QQ}$ lies in its orthogonal, the result follows.
\end{proof}

Given two weights $[a_k]_{k=1,\dots,s}$ and $[b_k]_{k=1,\dots,s}$ in $\Wt(\VV_{\bar \ud})$, we
have
$$[a_k]<_{\Fc} [b_k]\quad\iff\quad
\forall k=1,\dots,s\quad a_k\geq b_k\qquad {\rm and\ }[a_k]\neq [b_k].
$$

Adapting Algorithm~\ref{algo:Sreg} with
$T_\Fc$ replaced by $T_{\bar\ud}$ (and line \ref{line:Cp} replaced by ``$\codim \, \Span(S_0)=s$'')
requires some care on the implementation of Condition $(B'')$ given that $\Fc$ and hence $\dim(U(\tau))$ are not known.
The key to handle this is the following lemma
\begin{lemma}
  \label{lem:lowerbdimU}
Let $\tau\in X_*(T_{\ud})$ defining $\bar \ud$ such that $\bar \tau \in X_*(T_{\bar\ud})$ as above.
Let $\ue=(e_1\geq\cdots\geq e_s)$ be the nonincreasing reordering of $\bar \ud$.

Then $\dim U(\tau)\leqslant \left\lfloor\sum_{k=1}^s   \frac{d_k^2}2(1-\frac 1{e_k}) \right\rfloor$.
\end{lemma}

\begin{proof}
First, let us state $2$ combinatorial inequalities:
\begin{itemize}
\item Let $d,e\in \N^*$ with $e\leqslant d$. Let $\mathcal A_{d,e}=\{\underline{m}=(m_1, \dots, m_e)\in \N^e| \sum_i m_i=d\}$. 
As a consequence of the convexity of the square function, we have 
\begin{equation}
  \label{eq:combin1}
  \min_{\underline m\in \mathcal A_{d,e}}\left(\sum_i m_i^2\right)\geq \sum_{i=1}^e\left(\frac de\right)^2=
  \frac{d^2}{e}.
\end{equation}
\item Consider two tuples of nonnegative real numbers in nonincreasing order $x_1 \geqslant \dots \geqslant x_s$ and $y_1\geqslant \dots  \geqslant y_s$. We claim that
\begin{equation}\label{eq:combin2}
  \max_{\sigma \in \mathfrak S_s} \sum_k x_ky_{\sigma(k)}=\sum_k x_ky_k.
\end{equation} 
Indeed, assume that $\sigma\neq Id$, and let $i_0$ be minimal such that $\sigma(i_0)\neq i_0$. 
Letting $\sigma'=( i_0\;\; \sigma(i_0))\circ \sigma$, we have  
$\sum_k x_ky_{\sigma'(k)}-\sum_k x_ky_{\sigma(k)}=x_{i_0}y_{i_0}+x_{\sigma^{-1}(i_0)}y_{\sigma(i_0)}-x_{i_0}y_{\sigma(i_0)} -x_{\sigma^{-1}(i_0)}y_{i_0}=(x_{i_0}-x_{\sigma^{-1}(i_0)})(y_{i_0}-y_{\sigma(i_0)})$ 
which is nonnegative by the assumption on $i_0$.
By induction on $i_0$ the maximum is achieved for $\sigma=Id$.
\end{itemize}
We have $\dim U(\tau)=\frac 12 (\dim G_{\ud}-\dim G_{\ud}^{\tau})$. In each factor $\GL_{d_k}$ ($k\in [\![1,s]\!]$) of $G_{\ud}$, the stabilizer of $\tau$ is a Levi factor of the form $\prod_{i=1}^{\bar d_k}(\GL_{m_i})$ for some $(m_1\dots, m_{\bar d_k})\in \mathcal A_{d_k,\bar d_k}$. By \eqref{eq:combin1}, its dimension is lower bounded by $ {d_k^2}/{\bar d_k}$. The upper bound $\dim U(\tau)\leqslant \sum_k \frac{1}{2} (d_k^2-d_k^2/\bar d_k)=\sum_k(\frac{d_k^2}2(1-1/\bar d_k))$ follows.  To get the upper bound in terms of the $e_k$ instead of the $\bar d_k$, we apply \eqref{eq:combin2} with $x_k=d_k^2$ and $y_k=(1-\frac1{e_k})$.
\end{proof}

\bigskip
Let $\ud$ be a dimension vector. We now sum up the effective computation of
\[\Tau_{\ud}^{+}:=\left\{\tau\in X_*^+(T/Z) \textrm{ satisfying conditions \eqref{eq:3weakkron} and \eqref{eq:2veryweak}}\right\}.\]
For $u\in \N$ we let 
\[
\Tau_{ \ud}^{++}(u)=\left\{\tau\in X_*^{++}(T/Z) \,:\,
\begin{array}{c}
  \mbox{satisfying }\eqref{eq:3weakkron} \\ 
  \sharp\{\chi\in\Wt(\VV_{\ud})\,:\,\langle\chi,\tau\rangle>0\}\leq   u.
\end{array}\right\}.
\]

\begin{algorithm}[H]
    \caption{Computation of $\Tau_{\ud}^{+}$ in the Kronecker case}
    \label{algo:Tplus}
    \begin{algorithmic}[1]
     \State $\Tau^{+}:=\emptyset$ 
     \Statex\Comment{Initialization}
     \ForAll{$\ue\leqslant \ud$}    \label{line:ueud}
     \Statex\Comment{\emph{i.e.} $ \ue=\bar e_1\geqslant \cdots \geqslant e_s\geqslant 1$; $\forall i  e_i\leqslant d_i$}
     \State Compute $u:= \left\lfloor\sum_{k=1}^s   \frac{d_k^2}2(1-\frac 1{e_k}) \right\rfloor$ \Comment{see Lemma~\ref{lem:lowerbdimU}} 
    \ForAll{$\tau'\in \Tau_{\ue}^{++}(u)$}
    \Statex\Comment{computed with Algorithm~\ref{algo:Sreg} improved with Condition \eqref{eq:2veryweak}} \label{line:foralltauprime}
       \ForAll{$\tau \in X_*(T_{\ud}/Z)$ such that $\bar \tau$ is a permutation of $\tau'$ \label{line:extend}}
           \If{$\tau$ satisfies condition \eqref{eq:2veryweak} and \eqref{eq:3weakkron}}
          \State{Add $\tau$ to $\Tau^+$}
          \EndIf
\EndFor
\EndFor
\EndFor
    \end{algorithmic}
\end{algorithm}

\subsection{The fermionic (and bosonic) cases}

{\bf Symmetries.}
In the fermionic case $\VV=\bigwedge^rV$ with $\dim(V)=2r$, $\VV$ is
self dual as a $\SL(V)$-representation. As a consequence, the set of
inequalities is stable by $\tau\mapsto \left(\frac{|\tau|}
r-\tau_{2r+1-i}\right)$.

For any $r$, duality also gives a way to check the implementation by comparing $\bigwedge^rV$ with 
$\bigwedge^{\dim V-r}V$. 

\bigskip

\noindent {\bf Step 1.}
In fermionic and bosonic cases, there is a way to group faces similarly to what is done in the Kronecker case. 

Set $n=\dim(V)$. Given a face $\Fc$, we can define a partition $\lambda_{\Fc}$ of $n$ encoding the repetitions in the entries of any $\tau\in \Fc$.
The torus $T_\Fc$ as well as the multiplicity function $m_{\Fc}\,:\,\Wt_{T_\Fc}(\VV)\longto \NN)$ 
only depend on this partition. 

So, instead of running over the $2^{n-1}$ faces of $X_*^+(T)$, we run the analogue of line \ref{line:ueud} of Algorithm \ref{algo:Tplus} over the partitions of $n$.

\section{Refinements in Steps~3 and 4} %in Type A}

\subsection{Inversions set in type A}
\label{sec:raf3}

We assume here that $G=\GL_n$ and that $T$ is its subgroup of diagonal matrices.
A dominant one-parameter subgroup $\tau\in T$ yields a $\C^*$-action on $\C^n$ with weights $l_1>\dots>l_p$.
Denoting the weight space of weight $l_i$ by $V^{l_i}$, one gets $\Lie\;  U(\tau)=\bigoplus_{i<j} \Hom(V^{l_j}, V^{l_i})$ (upper block diagonal matrices).
For $w\in W^{P(\tau)}$, the inversion set $\Phi(w)$ corresponds to the data of one (bottom-left justified) partition for each block. 
The partition corresponding to the block $\Hom(V^{l_j}, V^{l_i})$ has length at most $\dim V_{l_i}$ and first part is at most $\dim V_{l_j}$. 

To construct the sets of inversions subject to our constraints, we perform a tree search, going through the matrix of blocks in such a way that $j-i$ increases (from the diagonal to the top corner).
When we reach the $(i,j)$ block, convexity and coconvexity constraints are determined by the partitions chosen in the $(i,k)$ and $(k,j)$ blocks (for various $k$). 
These constraints impose that the current partition $\lambda$ must satisfy $\lambda_{\rm in} \subset \lambda\subset \lambda_{\rm out}$ (where $\lambda_{\rm in}$ and $\lambda_{\rm out}$ are explicit). 
Moreover, the condition on the cardinality at level $l$ from \eqref{eq:wellcnum} imposes conditions on $|\lambda|$.
The choice of $\lambda$ replaces the choice of a single $\alpha\in \Phi_?$ in line \ref{line:choicealpha}, 8, 11 of Algorithm \ref{algo:W} while the partial computation of the partitions $\lambda_{in}$ and $\lambda_{out}$ in each forthcoming block plays the role of lines 6, 7, 9, 10.

In the Kronecker case, $G=G_{\ud}$ and we apply this method successively for each $\GL_{d_i}$. 
Constraints on multiplicities are propagated.

\subsection{Precomputation in Step 4}
\label{sec:raf4}

In this step, we compute the rank of matrices at some random points in some $\VV^\tau$. 
A first computational trick consists in fixing few random points $x_0$ in the whole $\VV$. We encode the tangent map of the orbit map $U\times \VV \rightarrow \VV$ at $(e,x_0)$ by an array of dimensions $(\dim \VV, \dim \VV, \dim U)$ with the coefficient of $(x_0)_{\chi}$ in position $(\chi, \chi', \alpha)$ whenever $\alpha + \chi=\chi'$

For each pair $(\tau, w)$ we choose as random point $x\in\VV^\tau$ the one obtained by slicing the coordinates of $x_0$ (since $\VV^\tau$ is a coordinate subspace). Then, the matrix of $T_{(e,x)}\pi$ is obtained by slicing the 3D-array and then summing along one axis. 

In our Python-Sage implementation, this precomputation of the 3D-array common to all pairs $(\tau,w)$ followed by the efficient slicing in each case helps to reduce significantly the time complexity.
\section{Refinements in birationality Step 5}
\label{sec:raf5}

As explained in Section~\ref{sec:Step5} our aim is to decide whether the ramification divisor of $\bar\pi$ 
is contracted or not. 

Recall that the contraction of the Schubert divisors of the form $D_{\beta}$ is checked via the computation of a rank of a tangent map $T_{(e,x)}\pi_v$. In particular, the precomputation suggested in Section~\ref{sec:raf4} also allows some time-saving in this context.

We now focus on the divisors related to $\pi$ itself. 
By Theorem~\ref{th:ram0}, we have to compute $\delta$, and then the reductions of $T_{\varphi(z)} J_{red}$ and $V_N(z)$ in 
the quotient ring $\CC[z]/(\delta)$. 
Some tricks avoid some of the most expensive intermediate computations in the naive approach regarding 
Section~\ref{sec:Step5} (see Remark~\ref{rem:Taylorutile}).

\begin{enumerate}
  %\item $J$ is the determinant of an explicit matrix $M=(\frac{\partial \pi_i}{\partial x_j})$.
  \item \label{compJbarv2}
 Up to a scalar,  the polynomial $ J_{red}$ can be computed using the following formula:
  $$
  J\propto J_{red}\times \gcd(J,\frac{\partial J}{\partial x_1},\dots,\frac{\partial J}{\partial x_N}).
  $$
  This method is faster than decomposing $J$ as a product of irreducible factors.
  \item \label{compdeltav2} Computing $\delta$. By substitution in $M$ one gets the matrix $ M_z$ of $T_{\varphi(z)}\pi$. 
  For $i=1,\dots,N$, let $\Delta_i$ be the gcd of the minors of the matrix $M_z$ of size $i$. 
  In particular, $J\circ\varphi=\Delta_N$. Recall that (up to a scalar)
  \begin{equation}
   \label{eq:si}
   \forall i=1,\dots,N\quad s_i=\frac{\Delta_i}{\Delta_{i-1}},
  \end{equation}
where $\Delta_0=1$. 
To compute $\delta$ with formulas~\eqref{eq:defdelta} it is sufficient to know $s_N$ and $s_{N-1}$. 
So, by formula~\eqref{eq:si}, it is sufficient to compute $\Delta_{N-1}$ and $\Delta_N=J\circ\varphi$. 
This method is faster than applying Smith's algorithm.
\item \label{compdeltav3} Still computing $\delta$. Actually, our program uses an alternative method to compute $\delta$ even faster, 
assuming that $\pi$ and $\varphi$ are defined over $\QQ$. 
First, recall from the proof of Theorem~\ref{th:ram0} that $\delta\in \Q[z]$ satisfies
$$
\delta=\prod_j(z-\zeta_j)
$$
where $\{\zeta_1,\dots,\zeta_s\}$ is the set of $z\in\CC$ such that $\rk( M_z(z))=N-1$.
Factor 
$$
\frac{\Delta_N}{\gcd(\Delta_N,\Delta_N')}=\delta_1\cdots\delta_p,
$$
with $\delta_i\in\QQ[z]$ irreducible.
Then,
$$
\delta=\prod_{j\in {\mathcal J}}\delta_j,
$$
where ${\mathcal J}$ is the set of $j$ such that the matrix $M_{z,\delta_j}$ obtained from $ M_z$ by 
reducing each coefficient in the number field $\QQ[z]/(\delta_j)$ has corank 1. 
This allows us to compute $\delta$ as follows.

\begin{algorithm}[H]
    \caption{Computation of $\delta$}
    \label{algo:delta}
    \begin{algorithmic}[1]
\State{Use the block triangular shape of $ M_z=M\circ\varphi$ to write 
$\Delta_N:=J\circ\varphi=\Delta_{N,1}\times\cdots\times\Delta_{N,p}$.
\Statex\Comment{\it This needs computations of determinants in $\QQ[z]$.}} \label{Jz}
\State{Factor each $\Delta_{N,i}$ in $\QQ[z]$.}
\State{Initialize the list $F_\delta$ of the factors of $\delta$ to the empty list.}
\State{For each factor $\delta'$ of some $\Delta_{N,i}$, add $\delta'$ to $F_\delta$ 
if $\delta'^2$ does not divide $\Delta_N$ or $ M_z$ modulo $\delta'$ has corank 1.
\Comment{\it if $\delta'$ appears in $s_{N-1}$ then $(\delta')^2|s_{N-1}s_N$}
\Statex\Comment{\it Here we compute the rank of matrices in a number field.}
}
\State{We get $\delta$ as a product of irreducible polynomials in $\QQ[z]$.}
   \end{algorithmic}
\end{algorithm}

\begin{remark}\label{Hilbert_irred}

Assume that $\pi$ is defined over $\QQ$ and use notation of the
proof of Theorem~\ref{th:ram0}. 
Consider an irreducible factor $h$ of $J$ over $\QQ$.

The associated divisor is either a sum of $D_i$'s or a sum of $E_i$'s.
Indeed, it cannot mix the two types of divisors since the condition
that distinguishes these two types is defined over $\QQ$.

Observe that, in the case of the $D_i$'s, Hilbert's irreducibility theorem (see e.g. \cite{FJ}), implies that
$h\circ \varphi$ is one of the factor $\delta'$.
\end{remark}

\item \label{compVNv2} Computing $V_N$. Smith's algorithm gives a way to compute $V_N$. We present here a faster method.
Think about $ M_z$ as an endomorphism of the $\CC[z]$-module $\CC[z]^N$ and consider the induced endomorphism 
$M_{z,\delta}$ of the $\CC[z]/(\delta)$-module $(\CC[z]/(\delta))^N$. 
By construction, the Kernel of $ M_{z,\delta}$ is a free $\CC[z]/(\delta)$-module of rank one generated by the class of $V_N(z)$.
Note that $M_{z,\delta}$ is also a $\CC$-linear map and that its Kernel has dimension 
$\deg(\delta)$ as a $\CC$-vector space. 
So $V_N$ can be determined by computing the kernel of a square matrix of size $(N\deg(\delta))$.
\item \label{compVNv3} Still computing $V_N$. Assume now that $\pi$ is defined over the rational numbers $\QQ$ and choose $(a,b)$ with 
coordinates in $\QQ$, too. 
Then, one can take for $V_N$ a generator of the Kernel of $M_{z,\delta,\QQ}$, the endomorphism of  $(\QQ[z]/(\delta))^N$.
Factor $\delta=\delta_1\dots \delta_p$ in $\QQ[z]$, with the $\delta_i$ irreducible (and pairwise distinct by construction).
We now use the Chinese remainder theorem:
\begin{equation}
\label{eq:chinois}
\QQ[z]/(\delta)\simeq \QQ[z]/(\delta_1)\times \cdots\times \QQ[z]/(\delta_p).
\end{equation}
Observe that, by the extended Euclidean algorithm, this isomorphism and its converse can be made perfectly explicit. 
Set $K_i=\QQ[z]/(\delta_i)$ that is a field. 
Then, $M_{z,\delta,\QQ}$ can be written as $M_{z,\delta,\QQ}=(M_1,\dots,M_p)$, where $M_i$ is an endomorphism of $K_i^N$.
From the $\QQ$-linear point of view, this amounts to a change of basis to write  $ M_{z,\delta,\QQ}$ as a bloc diagonal matrix. 
Now, compute a generator $V^i$ of the Kernel $M_i$ that is a $1$-dimensional $K_i$-vector space. 
Then, the pull-back of $(V^1,\dots,V^p)$ by the isomorphism~\eqref{eq:chinois} is $V_N$.  
\item\label{modulodeltai}
The next trick is to replace the condition of Theorem~\ref{th:ram0} by the equivalent one
\begin{equation}
  \label{condram0deltai}
  \forall \mbox{ irreducible factor }\delta_i\mbox{ of }\delta\qquad
  T_{\varphi(z)} J_{red}(V_N(z))=0\mbox{ in }\QQ[z]/(\delta_i) .
\end{equation}
Observe that by Algorithm~\ref{algo:delta}, the factors $\delta_i$'s are already computed.
First, this reformulation allows us to work only with matrices of coefficients in the number fields $K_i:=\QQ[z]/(\delta_i)$.
Secondly, splitting our condition into multiple ones allows us to stop checking at soon as one of the tests fails.
In some cases, this makes it unnecessary to compute the complete list of factors of $\delta$. 

\item Computing $T_{\varphi(z)} J_{red} \bmod \delta_i$.
Start by fixing a set $\xi_1,\dots,\xi_N$ of new variables that represent coordinates on tangent spaces. 
For any multi-index $\ualpha=(\alpha_1, \dots, \alpha_N)$, denote by  
 $D^{\ualpha}=\frac{1}{\alpha_1! \dots \alpha_N!}\partial^{|\ualpha|}/(\partial x_1)^{\alpha_1}\dots (\partial x_N)^{\alpha_N}$, the differential operator on $\QQ[x_1,\dots,x_N]$.
Here $|\ualpha|=\alpha_1+\cdots+\alpha_N$.
Set also  $ \xi^\ualpha=\prod_j\xi_j^{\alpha_j}$ and $
\mathcal D^sf(x,\xi)=\sum_{|\alpha|=s}  D^{\ualpha} f (x)\;  \xi^\ualpha
$, is a degree $s$ homogeneous polynomial in the variables $\xi_j$'s with coefficients in $\QQ[x_1,\dots,x_N]$.

It appears in the Taylor expansion of $f\in\QQ[x_1,\dots,x_N]$:
\[f( x+ \xi)=\sum_{s\in \N} \mathcal D^sf ( x, \xi). \]
For example, $T_{\varphi(z)} J_{red}=\mathcal D^1J_{red}\circ\varphi$, where the precomposition stands for the variables $x_i$'s.
For a product $f=f_1\dots f_p$, we have 
$$
\mathcal D^{s}f= \sum_{s_1+\dots+ s_p=s}\;\;\; \prod_{j=1}^p (\mathcal D^{s_i} f_j).
$$  
In particular, if $f=h^kg$, then 
\begin{equation} \label{partial_s}
\mathcal D^{s}f \in \left\{\begin{array}{ll} h\Q[x_1, \dots x_N] & \mbox{if }k<s\\ 
(\mathcal D^{1}h)^s \times g+h\Q[x_1,\dots,x_N]& \mbox{if }k=s \end{array}\right.
\end{equation}

Fix now an irreducible factor $\delta_i$ of $\delta$, and set $K_i=\QQ[z]/(\delta_i)$. 
Let $\ell$ denote the reduction of $\mathcal D^1J_{red}\circ\varphi$ modulo $\delta_i$.
It is a homogeneous polynomial of degree one in the $\xi_j$'s with coeffcients in $K_i$.
We aim to compute $\ell$ up to a nonzero constant in $K_i$ (see condition~\eqref{condram0deltai} to be checked).

By genericity of $\varphi$, $\delta_i$ divides exactly one irreducible factor of $J$ precomposed with $\varphi$. 
Thus, one can write $J=h^sg$ where $\delta_i$ divides $h\circ\varphi$ and does not divide $g\circ\varphi$. 
Note that Remark \ref{Hilbert_irred} implies that $\delta_i=h\circ\varphi$, but this is not used here. 

Since $h$ appears in $J_{red}$ with multiplicity one,  \eqref{partial_s} implies that 
$\ell$ is the reduction in $K_i$ of $\mathcal D^1h\circ\varphi$
up to a nonzero scalar in $K_i$. 

Again by \eqref{partial_s}, the class $\overline{\mathcal D^sJ\circ\varphi}$ of $\mathcal D^sJ\circ\varphi$ 
in $K_i[\xi_1,\dots,\xi_N]$ is $\ell^s$, up to a nonzero scalar in $K_i$.

Assume for a moment that $\overline{\mathcal D^sJ\circ\varphi}$  has been computed. 
Then, choose a variable $\xi_{j_0}$ such that $\xi_{j_0}$ (or equivalently $\xi_{j_0}^s$) appears in it. 
Now, $\ell$ is easy to deduce from the coefficients of the various monomials $\xi_{j_0}^{s-1}\xi_j$. 

\smallskip
It remains to compute $\overline{\mathcal D^sJ\circ\varphi}$. Fix $u$ in $w^{-1}U^-w\cap U$. By multiplication by $u$ on the first factor and by equivariance of $\pi$, we can identify $T_{(e,x)}$ with $T_{(u,x)}\pi$. For our maps $\pi$, observe that $ T_{(e,\bullet)}\pi=(x\mapsto T_{(e,x)}\pi)$ is linear. Denote by $M=M(x)$ the matrix of $T_{(e,x)}\pi$ in a fixed basis. We have $J=\det(M)$ and $J(x+\xi)=\det (M(x)+M(\xi))$, so by composition of Taylor series we get
\begin{equation}
 \label{Dsdet}\mathcal D^sJ(x,\xi)=\mathcal D^s(\det)(M(x), M(\xi)).
\end{equation}
It can easily be computed from the minors of $M$.  More precisely, in degree $s$, the monomials in each minor of size $s$ of $M(\xi)$ are coefficiented by the complementary minor of $M(x)$ of cosize $s$.

An important point is that this computation of $\mathcal D^sJ$ commutes with both the precomposition by $\varphi$ 
and the reduction modulo $\delta_i$. 
Consequently, $\overline{\mathcal D^sJ\circ\varphi}$ can be computed from the collection of cosize $s$ minors of the 
reduction of $T_x\pi$ in $K_i$. 
The most time-consuming step in the outlined method is this computation of minors of a matrix over $K_i$. 

\bigskip
This algorithm can be improved by observing that $\ell$ only depends on $h$. 
Indeed, if $M_{(k)}$ is any diagonal block of the matrix $M$ of $T_x\pi$ such that $\delta_i$ divides $\det(M_{(k)}\circ\varphi)$ , \emph{i.e.} $h$ divides $\det(M_{(k)})=:J_{(k)}$, 
then $M$ can be replaced by the smaller $M_{(k)}$ in the previous discussion. 

The main steps of this method can be summarized as follows. 

\begin{algorithm}[H]
\algblock[Name]{try}{EndTry}
\algtext*{EndTry}%
    \caption{Computing $\ell\equiv T_xJ_{red}\circ \varphi \pmod {\delta_i}$ up to scalar}
    \label{algo:ell}
    \begin{algorithmic}[1]
\State{Consider the diagonal blocks $M_{(k)}$ of $T_x\pi$ such that $\delta_i$ is an irreducible factor  of $\det M_{(k)}\circ\varphi$.} 
\State{Compute the maximal integers $s_k$ such that $\delta_i^{s_k}$ divides $\det M_{(k)}\circ\varphi$.}
\State{Choose $M_{(k)}$ so that $s=s_k$ is minimal, and subsequently so that size$(M_{(k)})$ is minimal.}
\State{Reduce $M_{(k)}\circ\varphi$ modulo $\delta_i$ to get $\bar M_{(k)}$ with entries in $K_i$.}
\State{Compute all the size$(M_{(k)})-s$ minors of $\bar M_{(k)}$.}
\State{Deduce $\overline{\mathcal D^sJ_{(k)}\circ\varphi}$.} \Comment{using discussion below \eqref{Dsdet}}
\State{Extract $\ell$ from the relation $\overline{\mathcal D^sJ_{(k)}\circ\varphi}=\ell^s$.}
   \end{algorithmic}
\end{algorithm}

\end{enumerate}
\begin{remark}\label{rk:component_Vtauneg}
In \eqref{Dsdet}, we have exploited the fact that $T_{(e,\bullet)}\pi$ is linear. %For other maps than $This is not true for our orbit maps $\pi$ in general, but we can reduce ourselves  to this case, by considering only the points of the form $(e,x')$ by equivariance.
%to this case since  $J = \det\circ (T_{\bullet}\pi)= \det\circ (T_{(e,\bullet)}\pi)\circ pr_{V^{\tau\leq0}}$ ($\pi$ is $w^{-1}U^{-}w\cap U$-equivariant so the ramification locus does only depend on the $V^{\tau \leq0}$-component) while  $x\mapsto (T_{(e,x)}\pi)$ is an affine map.\\
For maps other than $\pi$, we can adapt formula~\eqref{Dsdet} by composing Taylors series.
\end{remark}

\begin{remark}
\label{rem:Taylorutile}
The improvements presented in this section are essential to deal with examples beyond $Kron(5,5,5)$. 
Indeed, the computation of $J$ alone quickly becomes problematic, as it involves calculating the determinant 
of a matrix whose entries are linear forms in $\dim\VV^\tau$ variables.
Evaluating this determinant in the ring $\Q[x_1,\dots,x_N]$ soon proves infeasible.
In contrast, Algorithm~\ref{algo:ell} carries out all calculations in a number field.
\end{remark}

\section{An example}
\label{sec:exple}

Consider the cone Kron$(3,3,3,1)$ with $\tau=(2\ 1\ 0\ |\ 2\ 1\ 0\ |\ 3\ 2\ 0\ |\ -4)$ and $w=(w_0,w_0,w_0)$. 

The map $\pi$ is 
$$
\pi\,:\,U\times \VV^{\tau\leq 0}\longto \VV.
$$
By convention here, we start all indices at $0$ so that $[0,0,0]=000$ is the highest weight of $\VV$ and $222$ is its lowest weight. The action of $\tau$ on the weights of $\VV$ is given by
$$
\scalebox{0.9}{$\begin{array}{|c|c|c|c|c|c|c|c|}
\hline
3&2&1&0&-1&-2&-3&-4\\
\hline
000&001,&011, 020, 101,&002,021,111,&012,102,121,&022,112,&122,&222\\
&010,100&110, 200&120,201,210&211, 220&202, 221&212&\\
\hline
\end{array}$}
$$
One can check that the 6 weights of $\VV^\tau$, that is orthogonal to $\tau$, span a 6-dimensional space.
Here, $\tau$ is regular and $K^\tau$ is the maximal compact subgroup. 
Hence, the condition~\ref{ass:Kron3} is satisfied since it is equivalent to \eqref{eq:3weak}.
Moreover, $\VV^{\tau>0}$ is isomorphic to $\lu$ as a $\tau$-module, according to  
\eqref{eq:wellcnum} with $(w_0,w_0,w_0)$. 
So, Steps 1 to 3 are passed by this example. 

At a point $(e,x)$ where $x\in\VV^{\tau\leq 0}$ with coordinates $v_{002},v_{021}\dots$, the matrix of
$T_{(e,x)}\pi$ is of the form
$$
M=\begin{pmatrix}
  A&0&0\\
  B&I_{6}&0\\
  C&0&I_{12}
\end{pmatrix}
$$
where the blocks in columns correspond respectively to $\lu,\VV^\tau,\VV^{\tau<0}$ and 
blocks in rows correspond to  $\VV^{\tau>0},\VV^\tau,\VV^{\tau<0}$. 
Moreover, here 
$$
A=
\scalebox{0.9}{$\begin{pmatrix}
  v_{002} & 0 & 0 & 0 & 0 & 0 & 0 & 0 & 0 \\
0 & v_{201} & v_{021} & v_{002} & 0 & 0 & 0 & 0 & 0 \\
v_{012} & v_{210} & 0 & 0 & 0 & 0 & 0 & 0 & 0 \\
v_{102} & 0 & v_{120} & 0 & 0 & 0 & 0 & 0 & 0 \\
0 & v_{211} & 0 & v_{012} & v_{111} & 0 & 0 & v_{021} & 0 \\
v_{022} & v_{220} & 0 & 0 & v_{120} & 0 & 0 & 0 & v_{021} \\
0 & 0 & v_{121} & v_{102} & 0 & v_{201} & v_{111} & 0 & 0 \\
v_{112} & 0 & 0 & 0 & 0 & v_{210} & 0 & v_{120} & v_{111} \\
v_{202} & 0 & v_{220} & 0 & 0 & 0 & v_{210} & 0 & v_{201}
\end{pmatrix}$}
$$
and
$$
B=
\scalebox{0.9}{$\begin{pmatrix}
  0 & v_{202} & v_{022} & 0 & v_{102} & 0 & v_{012} & 0 & 0 \\
  0 & v_{221} & 0 & v_{022} & v_{121} & 0 & 0 & 0 & 0 \\
  0 & 0 & 0 & v_{112} & 0 & v_{211} & 0 & v_{121} & 0 \\
  v_{122} & 0 & 0 & 0 & 0 & v_{220} & 0 & 0 & v_{121} \\
  0 & 0 & v_{221} & v_{202} & 0 & 0 & v_{211} & 0 & 0 \\
  v_{212} & 0 & 0 & 0 & 0 & 0 & 0 & v_{220} & v_{211}
\end{pmatrix}$}.
$$

$A$ is block triangular with blocks of size $1,3,5$ (that is, the multiplicities of the action of $\tau$ on $\VV^{\tau>0}$).
Denote by $A_1$, $A_2$ and $A_3$ the diagonal blocks. 
Thus,
$$
J=\det M=\det A = \prod_i \det A_i= 3 v_{002}^{2} v_{021} v_{111} v_{120}^{2} v_{201} v_{210}^{2},
$$
and
$$
 J_{red}=3v_{002} v_{021} v_{111} v_{120} v_{201} v_{210}.
$$
The fact that $J$ is nonzero ensures that the inequality passes step 4 in symbolic method.\\

Here, in the partial compactification of $\pi$, we add 6 Schubert divisors. 
These 6 are associated to simple roots. 
By Lemma~\ref{lem:posalphaR} these divisors are not contained in the ramification locus. 
Thus, the first step of Step~5  is trivial.

Observe that the $A_i$, and hence $J$, only depend on variables of $\VV^\tau$ as expected by Lemma~\ref{lem:wellc}. 
Equivalently, the ramification divisor is $U\times R_0\times V^{\tau<0}$ for some divisor $R_0$ in $\VV^\tau$.
The gradient of $J_{red}$ in coordinates $\dots, v_{002},v_{012}\dots$ is the row vector
$$
\begin{array}{r r c c}
  L=3(0\;\;, \qquad \dots\qquad0\;\;, &v_{021} v_{111} v_{120} v_{201} v_{210}, &  v_{002} v_{111} v_{120} v_{201} v_{210}, \\
  v_{002} v_{021} v_{120} v_{201} v_{210}, &
  v_{002} v_{021} v_{111} v_{201} v_{210}, & 
   v_{002} v_{021} v_{111} v_{120} v_{210}, \\ v_{002} v_{021} v_{111} v_{120} v_{201},&0\qquad\qquad,&\qquad \quad \dots \quad \qquad 0).
\end{array}
$$

Thus, $R_0$ is the union of the 6 coordinate hyperplanes. 
After substituting $v_{002}$ by $0$ in $A$, one gets a matrix $\bar A$ of rank $8$ in the fraction field. 
This means that the component ${\operatorname{div}}(v_{002})$ is one $D_i$ (notation of the proof of Theorem~\ref{th:ram0}).
Working similarly with the 6 components we get that there are no $E_k$ here.

Similarly, define $\bar M$ and $\bar L$ by substituting $v_{002}$ by $0$.
One can check that the kernel of $\bar M$ does not lie in the kernel of $\bar L$.
This means that the component ${\operatorname{div}}(v_{002})$ is not contracted hence the inequality $\langle w\tau ,\cdot \rangle\leq 0$ does not define a facet of the moment cone. 
Actually, the 5 other components are contracted in this case.

\section{Filter 1: Belkale-Kumar-Ressayre's theorem}
\label{sec:BKR}

\subsection{Introduction}
    
    In \cite{BKR}, in a context different but similar to 
Theorem~\ref{th:Kron}, it is shown that certain multiplicities for the
representation theory of Levi $G^{\tau}$ are equal to 1.
In this section, we prove a variant of \cite[Theorem~1.4]{BKR} in
our affine setting. We use this result as a filter allowing us to 
eliminate some inequalities. 
This filter has two advantages over ramification: it is deterministic
and faster. Unfortunately, some superfluous inequalities pass through the filter. 
For example, for Kronecker $(4,4,4)$, 53 inequalities are eliminated on 183 truly superfluous.

\subsection{A statement}

Let $G$ be a connected reductive group acting on a representation $\VV$.

\begin{theo}
  \label{th:BKR}
  Assume that $\CC[\VV]^T=\CC$.
  Let $\tau$ be a dominant one-parameter subgroup of $T$ and $w\in
  W^{P(\tau)}$.
  Set
  \begin{equation}
    \label{eq:defchidet}
   \chi_{det}=-\sum_{\alpha\in \Phi(w)}\alpha+\sum_{
     \begin{subarray}{c}
       \chi\in \Wt(\VV)\\
       \langle\chi,\tau\rangle>0
     \end{subarray}}
   \chi.
  \end{equation}
  Assume that the map
$$
\begin{array}{cccl}
  \pi\,:\,&(w\inv U^-w\cap U)\times \VV^{\tau\leq 0}&\longto&\VV\\
&(g,v)&\longmapsto&gv
\end{array}
$$ 
satisfies Conditions \ref{ass:Kron1} and \ref{ass:Kron2} of Theorem~\ref{th:Kron}.

    Let $\bar w\in W^{G^{\tau}}$ be defined by $w\inv B^-w\cap
    G^{\tau}=\bar w\inv (B^-\cap G^{\tau})\bar w$.

    Then, for any $n\geq 1$, we have
    \begin{equation}
    \dim\bigg(\CC[\VV^{\tau}]^{(B^-\cap
      G^{\tau})_{n \bar w\chi_{\det}}}
    \bigg)=1.\label{eq:dimVL}
  \end{equation}
\end{theo}

\bigskip
\begin{proof}
  We first observe that the restriction of $\det(T\pi)$ provides a nonzero element in the space
  of \eqref{eq:dimVL} for $n=1$. Then it is sufficient to prove that,
  for any $n$, this space is contained in a one dimensional space.
  We make this by push-forwarding the situation on $\VV$ using
  $\pi$. This will work because the out of control part of $\VV$ has
  codimension at least two, and hence does not contribute to the
  regular functions.

Consider
  $$
\bar\pi\,:\,\overline{w\inv B^-wP(-\tau)}\times_{w\inv B^-w\cap P(-\tau)}\VV^{\tau\leq 0}\longto\VV,
$$
as in Section~\ref{sec:bircompact}.
Recall that $\bar\pi$ is proper. But it is also dominant, so it is
surjective.

We now consider several restrictions of $\bar\pi$. Set

$$
\begin{array}{ll}
  X_w=w\inv \overline{B^-wP(-\tau)/P(-\tau)}&\mbox{the translated
                                                 Schubert variety;}\\[1em]
  X_w^s&\mbox{the smooth locus
         of }X_w;\\[1em]
  X_w^h=w\inv Q wP(-\tau)/P(-\tau)&\mbox{the homogeneous
                                                  locus of } X_w,\\[1em]
  &\mbox{where }Q=\{g\in G\,:\,g \overline{B^-wP(-\tau)}\subset
 \overline{B^-wP(-\tau)}\};\\[1em]
  X_w^\circ=w\inv B^- wP(-\tau)/P(-\tau)&\mbox{the translated
                                                 Schubert cell.}
\end{array}
$$

Consider
$$
\begin{array}{cccl}
\bar\pi^\dag\,:&\Xc^\dag:=\tilde X_w^\dag\times_{
              P(-\tau)}\VV^{\tau\leq 0}&\longto&\VV\\
            &[g:v]&\longmapsto&gv,
\end{array}
$$
where $\dag=s,h,\circ$ or is empty and $\tilde X$ denotes the pullback
in $G$ of $X\subset G/P(-\tau)$ by $G\longto G/P(-\tau)$.

\bigskip
Recall that $\bar\pi^\circ=\pi$. Consider now the boundary divisors
$D_\beta$ and the determinant bundle
$$
\Dc:=\bigwedge^N\Tc^*\Xc^s\otimes (\eta^s)^*(\bigwedge^N\VV),
$$
where $N=\dim(\VV)$.

\bigskip
The determinant of $T\bar\pi^s$ is a nonzero  section
$\sigma_0$ of $\Dc$.
Set $Q'=w\inv Qw$. 
Since $\Rc^s$ is stable by the parabolic subgroup $Q'$, there exists a
unique $Q'$-linearisation of $\Dc$ such that $\sigma_0$ is
$Q'$-invariant. 
If necessary, we replace $G$ by its universal covering. 

For any positive $n$, the map
$$
\begin{array}{cccl}
\iota_n\,:&H^0(\Xc^s,\Dc^{\otimes n})&\longto&\CC(\Xc^s)\\
&\sigma&\longmapsto&\frac{\sigma}{\sigma_0^n}
\end{array}
$$
is a $Q'$-equivariant embedding. The image of $\iota_n$ is
$$
H^0(\Xc^s,n\Rc^s)=\{f\in\CC(\Xc^s)\,:\,\div(f)+n\Rc\geq 0\},
$$
and is contained in $\CC[\Xc^h-\Rc^h]$.

\bigskip
Set $\Omega=\bar\pi^h(\Xc^h-\Rc^s)$ be the open subset of $\VV$ and $Z$ be
the union of the components of $\Xc-\Xc^h$ that have codimension at
least two. 
Since $\bar\pi$ is surjective, Lemma~\ref{lem:posbetaR} implies that the complement $\VV-\Omega$ is contained in 
$$
\bar\pi(\Rc)\cup \bar\pi(Z).
$$
By Zariski's main theorem, $\bar\pi(\Rc)$ has codimension at least
two. As a consequence $\VV-\Omega$ has codimension at least
two and, by Hartog's lemma, $\CC[\Omega]=\CC[\VV]$.

Still by Zariski's main theorem, the restriction of $\eta^h$ to
$\Xc^h-\Rc^s$ is an isomorphism on $\Omega$. Let
$\zeta\,:\,\Omega\longto\Xc^h$ denote its converse.
 Then, the linear map
$$
\zeta^*\circ\iota_n\,:\,H^0(\Xc^h,\Dc^{\otimes n})\longto \CC[\Omega]=\CC[\VV]
$$
is both injective and $Q'$-equivariant.
Since $\CC[\VV]^T=\CC$, this implies that 
\begin{equation}
  \label{eq:dimH0Q}
  \dim(H^0(\Xc^h,\Dc^{\otimes n})^{Q'})\leq 1.
\end{equation}

With this inequality, the following lemma ends the proof:

\begin{lemma}
\label{lem:isoH0Q}
The restriction map induces an isomorphism
$$
H^0(\Xc^h,\Dc^{\otimes n})^{Q'}\simeq \CC[\VV^{\tau}]^{(B^-\cap
      G^{\tau})_{ n \bar w.\chi_{\det}}}.
$$
\end{lemma}

\begin{proof}
We first observe that the restriction of $\Dc$ to $\VV^{\tau\leq 0}$
is trivial as line bundle and linearized by the character
$\chi_{\det}$. 
But the determinant of $T\pi$ is a $\tau(\CC^*)$-invariant section
of $\Dc$ on $\Xc^0$. By assumption, it restricts as a $\CC^*$-equivariant function on $\VV^{\tau\leq 0}$ of weight
$\langle\chi_{\det},\tau\rangle$ that does not vanish identically on
$\VV^{\tau}$. 
Since $\tau$ acts trivially on $\VV^{\tau}$, we deduce that 
\begin{equation}
  \label{eq:chidettau0}
  \langle\chi_{\det},\tau\rangle=0.
\end{equation}

Recall that $\Xc^h=Q'\times_{Q'\cap P(-\tau)}\VV^{\tau\leq 0}$. 
Then
$$
H^0(\Xc^h,\Dc^{\otimes n})^{Q'}\simeq
H^0(\VV^{\tau\leq 0},\Dc^{\otimes n})^{Q'\cap P(-\tau)}.
$$ 
Using \eqref{eq:chidettau0}, \cite{GITEigen} implies that 
$$
H^0(\VV^{\tau\leq 0},\Dc^{\otimes n})^{Q'\cap P(-\tau)}\simeq
H^0(\VV^{\tau},\Dc^{\otimes n})^{Q'\cap G^{\tau}}.
$$
Since, on $\VV^{\tau}$, the restriction of $\Lc$ is the trivial
bundle linearized by $\chi_\det$, 
$$
H^0(\VV^{\tau},\Dc^{\otimes n})^{Q'\cap G^{\tau}}\simeq
\CC[\VV^{\tau}]^{(w\inv Q w\cap
      G^{\tau})_{ n \chi_{\det}}}.
$$
But $w\inv Q w\cap G^{\tau}$ is a parabolic subgroup of $G^{\tau}$
containing $w\inv B^-w\cap G^{\tau}$ and on which $\chi_{\det}$
extends. Then, it is well known that 
$$
\CC[\VV^{\tau}]^{(w\inv Q w\cap G^{\tau})_{ n \chi_{\det}}}=
\CC[\VV^{\tau}]^{(w\inv B^- w\cap
      G^{\tau})_{ n \chi_{\det}}}.
$$
By the definition of $\bar w$, we have
$$
\CC[\VV^{\tau}]^{(w\inv B^- w\cap
      G^{\tau})_{ n \chi_{\det}}}=
\CC[\VV^{\tau}]^{(\bar w\inv (B^-\cap
      G^{\tau})\bar w))_{ n \chi_{\det}}}\simeq
\CC[\VV^{\tau}]^{(B^-\cap
      G^{\tau}))_{ n \bar w\chi_{\det}}}.
$$
The lemma is obtained by composing all these five isomorphisms.
\end{proof}
\end{proof}

\subsection{The multiplicity (\texorpdfstring{\ref{eq:dimVL}}{eq:dimVL}) in the Kronecker case}

Let us come back to the situation of the Kronecker cone:
fix $\ud=(d_1,\dots,d_s)$, $\VV=V_1\otimes\cdots\otimes V_s$ and
$G_{\ud}=\GL(V_1)\times\cdots\times\GL(V_s)$.
Fix $(\tau,w)$ such that $\pi$ is dominant.
Recall from Section~\ref{sec:step1math}, the definitions of $\bar \ud$
and $\bar \tau$. Identify  $X_*(T)$ with $\ZZ^{d_1}\oplus\cdots\oplus\ZZ^{d_s}$.
Let $m\tau\in \NN^{\bar d_1}\oplus\cdots\oplus\NN^{\bar d_s}$ whoose
the entries are the number of repetition of the corresponding entry of
$\bar\tau$ in $\tau$.

\begin{exple}\label{exple:tau}
  Here, $\ud=(5,5,5,1)$ and
  $$
  \tau=(3,2,2,-1,-1\,|\,2,2,-3,-3,-3\,|\,3,-1,-1,-1,-2\,|\,5).$$
  Then $\bar\tau=(3,2,-1\,|\,2,-3\,|\,3,-1,-2\,|\,5)$,
  $\bar \ud=(3,2,3)$
  and $m \tau = (1,2,2\, |\,2,3\, \allowbreak |\,1,3,1)$.
\end{exple}

\bigskip
For $k=1,\dots,s$ and $1\leq i\leq \bar d_k$, define $\bar\tau[i,k]$
to be the $i^{\rm th}$-entry of the $k^{\rm th}$-component of
$\tau$. On the example $\bar\tau[2,2]=-3$ and $\bar\tau[3,2]=-1$.

Observe that for each $k=1,\dots,s$, the space $V_k$ decomposes as
$$
V_k=\sum_{i=1}^{\bar d_k} V_k^{\bar\tau[i,k]},
$$
where $V_k^{\bar\tau[i,k]}$ is the $m\tau[i,k]$-dimensional subspace
on which $\tau$ acts with weight $\bar\tau[i,k]$.

Set
$$
\Pc:=\{(i_1,\dots,i_s) \;:\; \sum_{k=1}^s \bar\tau[i_k,k]=0\}.
$$
Then
\begin{equation}
  \label{eq:4}
  \VV^{\tau}=\bigoplus_{(i_1,\dots,i_s) \in \Pc} \bigotimes_{k=1}^s V_k^{\bar\tau[i_k,k]}.
\end{equation}
Moreover,
\begin{equation}
  \label{eq:1}
  G_{\ud}^{\tau}\simeq \prod_{i,k \ s.t.\ i\leq \bar d_k}\GL(V_k^{\bar\tau[i,k]}).
\end{equation}

In particular, a polynomial irreducible representation of $G_{\ud}^{\tau}$
is the data of a tabular $\Nu$ of partitions indexed by pairs $(i,k)$ such that
$$
\forall k=1,\dots,s\quad\forall i=1,\dots,\bar d_k\quad \ell(\Nu[i,k])\leq m\tau[i,k].
$$
An example of $\Nu$ for the $\tau$ as in Example~\ref{exple:tau}:
$$
\Nu=
\begin{array}{cccc}
  3&2\geq 1&2\\
  4\geq1&3\geq 2\geq 1&4\geq 2\geq 2&6\\
  3\geq 1&&3
\end{array}
$$
Let us fix such a $\Nu$.

\smallskip

Notation $\Kron(\Lambda)$: Set $p=\sharp\Pc$. 
Let  $\Lambda=(\Lambda[j,k])_{1\leq j\leq p, 1\leq k\leq s}$ be a
rectangular table with partitions as entries. To each row, we associate the multiple
Kronecker coefficient associated to the $s$ partitions on it. 
We denote by $\Kron(\Lambda)$ the product of the so obtained $p$
Kronecker coefficients associated to the $p$ rows.\\

\smallskip 

Notation $\LR(\Nu,\Lambda)$:
Number the elements of $\Pc$ from $1$ to $p$.
Think about $\Pc$ as a matrix with entries in $\NN$, $p$ rows and $s$
columns.
The entry $\Pc[j,k]$ is the $k^{th}$ entry of the $j^{th}$ elements of
$\Pc$.

Fix a  position $(i,k)$ in $\Nu$ (that is, $1\leq k\leq s$ and $1\leq
i\leq\bar d_k$).
Consider the set 
$$
S_{i,k}:=\{(j,k)\,:\,\Pc[j,k]=i\}
$$ 
of positions in the column $k$ of $\Pc$ with $i$ as entry. 
Let $c_{i,k}$ denote the multiple Littlewood-Richardson coefficient
with $\Nu(i,k)$ as outer partition and $\{\Lambda[j,k]\,:\, (j,k)\in
S_{i,k}\}$ as inner partitions. 
Denote by $\LR(\Nu,\Lambda)$ the product of these $\sum_{k=1}^s\bar
d_k$ multiple LR-coefficients.

\begin{prop}
  \label{prop:multCVtauKron}
The $G_{\ud}^{\tau}$-irreducible representation of highest weight $\Nu$
appears in $S^\bullet \VV^{\tau}$ with multiplicity
$$
\sum_{\Lambda} \Kron(\Lambda).\LR(\Nu,\Lambda),
$$
where the sum runs over the $(\sharp\Pc \times s)$-tabulars of partitions 
$\Lambda$ satisfying
\begin{enumerate}
\item $\forall 1\leq j\leq \sharp\Pc\quad
\forall 1\leq k<k'\leq s\quad
|\Lambda[j,k]|=|\Lambda[j,k']|=:\delta(j)$;
\item $\forall  k=1,\dots,s\quad\forall i=1,\dots,\bar d_k\quad
\sum_{j {\rm\ s.t.\ }\Pc(j,k)=i}\delta(j)=|\Nu[i,k]|$;
\item $\forall 1\leq j\leq \sharp\Pc\quad
\forall 1\leq k<k'\leq s\quad\ell(\Lambda[j,k])\leq m\tau(\Pc[j,k],k)$.
\end{enumerate}
\end{prop}

\begin{proof}
  By \eqref{eq:4}, we have:
  $$
S^\bullet \VV^{\tau}=\bigoplus_{\underline\delta\in\Pc^\NN}\bigotimes_{p\in\Pc}S^{\underline\delta(p)}\VV(p).
$$
By \eqref{eq:SWdualmK}, we get
$$
S^\bullet
\VV^{\tau}=\bigoplus_{\underline\delta\in\Pc^\NN}\bigotimes_{p\in\Pc}
\bigoplus_{
    (\lambda_1,\dots,\lambda_s)\in \Lambda(p,\underline\delta)
}g_{\lambda_1,\dots,\lambda_s}
\bigotimes_{k=1}^s S^{\lambda_k}V_k^{p_k},
$$
where
$$
\Lambda(p,\underline \delta)=\{(\lambda_1,\dots,\lambda_s)\,|\,
    \ell(\lambda_i)\leq m\tau(p_k,k)\mbox{ and }
    |\lambda_i|=\underline \delta(p)\ \forall i\}.
    $$
    Set
    $$
    \Lambda(\underline \delta)=\prod_{p\in\Pc}\Lambda(p,\underline
    \delta).
    $$
    An element $\tilde\lambda\in \Lambda(\underline \delta)$ is the
    data for each entry $(j,k)$ of $\Pc$ of a partition
    $\tilde\lambda[j,k]$ satisfying:
    \begin{itemize}
    \item $|\tilde\lambda[j,k]|=\underline\delta(j)$ for any $j$; and
      \item $\ell(\tilde\lambda[j,k])\leq m\tau(j,k)$ for any $j$.
    \end{itemize}
    We now commute $\oplus$ and $\otimes$:
\[
S^\bullet
\VV^{\tau}=\bigoplus_{\underline\delta\in\Pc^\NN}
\bigoplus_{\tilde\lambda\in \Lambda(\underline \delta)}
\bigotimes_{
  \begin{subarray}{c}
  p\in\Pc {\rm\ with}\\
  {\rm index\ }i
  \end{subarray}
}
g_{\tilde\lambda[i,1],\dots,\tilde\lambda[i,s]}
\bigotimes_{k=1}^s S^{\tilde\lambda[i,k]}V_k^{p_k}.
\]
The inner tensor product in this formula can be rewritten as
$$
\prod_i g_{\tilde\lambda[i,1],\dots,\tilde\lambda[i,s]}\times
\bigotimes_{(i,k) {\rm\ position\ in\ }\Pc}S^{\tilde\lambda[i,k]}V_k^{\Pc[i,k]}.
$$
We now group the terms associated to a given pair $(\Pc[i,k],k)$:
$$
\bigotimes_{(i,k) {\rm\ position\ in\
  }\Pc}S^{\tilde\lambda[i,k]}V_k^{\Pc[i,k]}=
\otimes_{(i,k){\rm \ s.t.\ } i\leq\bar d_k}\otimes_{j{\rm \
    s.t.\ } \Pc[j,k]=i} S^{\tilde\lambda[j,k]}V_k^i.
$$
But
$$
\otimes_{j{\rm \
    s.t.\ } \Pc[j,k]=i} S^{\tilde\lambda[j,k]}V_k^i
=\bigoplus_{\nu(i,k)\in\Nu(\underline\delta)}
c^{\nu(i,k)}_{\tilde\lambda[\bullet,k]}S^\nu V_k^i,
$$
where
$$
\Nu(\underline\delta)=\left \{(\nu(i,k))_{i,k} {\rm\ partitions\ }
\left|
  \begin{array}{l}
    i\leq\bar d_k\\
    \ell(\nu(i,k)\leq m\tau[i,k]\\
    |\nu(i,k)|=\sum_{j{\rm \
    s.t.\ } \Pc[j,k]=i} \underline\delta(j)
  \end{array}
\right .
\right\}.
$$
Finally, we get
$$
S^\bullet
\VV^{\tau}=\bigoplus_{\Nu}\bigoplus_{\Lambda}
\LR(\Nu,\Lambda)\Kron(\Lambda) S^\Nu\VV^\tau.
$$
The proposition follows.
\end{proof}

\bigskip
Proposition~\ref{prop:multCVtauKron} gives an algorithm to
compute the multiplicity:

\begin{algorithm}[H]
\algblock[Name]{try}{EndTry}
\algtext*{EndTry}%
    \caption{Computing determinantal multiplicities of $G^\tau$}
    \label{algo:multGtau}
    \begin{algorithmic}[1]
\State{Compute $\bar\tau$, $m\tau$, $\bar d$ and $\Pc$.} 
\Comment{This step only depends on $\tau$}
\State{Find the maps $\delta\,:\,[p]\longto\NN$ satisfying 2.}
\Statex\Comment{It is the set of integer points in a polytope. SageMath has
an implemented algorithm to enumerate these points. 
}
\State{For each such $\delta$ enumerate the $\Lambda$ satisfying 1. and
  3.}
\State{Compute Kron and LR.}
\Statex\Comment{Using Formulas~\eqref{eq:assKron} and \eqref{eq:assLR}, it is
sufficient to compute Kronecker and LR coefficients for three
partitions. 
For LR, we use Buch's program in Sage. For Kron, we use
SymmetricFunctions.}
   \end{algorithmic}
\end{algorithm}

\subsection{The multiplicity \texorpdfstring{\eqref{eq:dimVL}}{(eq:dimVL)} in the fermionic and bosonic case}

Consider now the fermionic representation $\VV=\bigwedge^rV$ or the bosonic
representation $\VV=S^rV$.
Let $\theta(r)$ denote the partition $1^r$ in the fermionic case and $r$
in the bosonic case, in such a way that $\VV=S^{\theta(r)}V$ in both case.

\begin{prop}
  \label{prop:multCVtauFermion}
  In fermionic or bosonic case,
the $G^{\tau}$-irreducible representation of highest weight $\Nu$
appears in $S^\bullet \VV^{\tau}$ with multiplicity
$$
\sum_{\delta}\sum_{\Mu}\sum_{\Lambda} \Kron(\Lambda).\LR(\Nu,\Mu)\prod_{j=1}^sa(\Lambda(I,j),\theta(I(j)),\Mu(I,j))
$$
where $\delta\,:\,\Pc\longto \NN$ and 
$\Mu$ and $\Lambda$ are maps from 
$
\Pc\times \llbracket s\rrbracket
$
to the set of partitions such that 
\begin{enumerate}
\item $\forall j\in\llbracket s\rrbracket$, we have
  $$
\sum_{I\in\Pc}I(j)\delta(I)=|\Nu(j)|.
  $$
\item $\forall I\in\Pc\quad
  \forall  j\in\llbracket s\rrbracket$,
  $$
  |\Lambda(I,j)|=\delta(I)\quad{\rm and}\quad
  \ell(\Lambda(I,j))\leq\dim(S^{\theta(I(j))}\CC^{m_j}).
$$
\item $\forall I\in\Pc\quad
  \forall  j\in\llbracket s\rrbracket$,
  $$
  |\Mu(I,j)|=I(j)\delta(I)\quad{\rm and}\quad
  \ell(\Mu(I,j))\leq m(j),
  $$
  \end{enumerate}
and
$$
\LR(\Nu,\Mu)=\prod_{j=1}^s c^{\Nu(j)}_{\Mu(\bullet,j)}\quad{\rm
  and}\quad
\Kron(\Lambda)=\prod_{I\in\Pc}g_{\Lambda(I,\bullet)}.
$$

\end{prop}

\begin{proof}
  We do not give the details, but the proof is similar to that of
  Proposition~\ref{prop:multCVtauKron} using
  
\[
  \wedge^k(E \oplus F) \cong \bigoplus_{i=0}^k \left( \wedge^i E
    \otimes \wedge^{k-i} F \right)
\]

and

\[
  S^n(E \oplus F) \cong \bigoplus_{k=0}^n \left( S^k E \otimes S^{n-k}
    F \right).
\]
\end{proof}

\section{Two filters: Gröbner and linear triangular}
\label{sec:2filters}

In this section, we describe two methods to detect whether $\pi$ is birational. 
Both methods use defining equations of the fibers of $\pi$.

The first method works via a computation of a Gröbner basis of a fiber. 
Unfortunately, the computational complexity of this method is too high to handle a significant amount of cases, even with a probabilistic method. 
For Kron(5,5,5), with 1s limit per $(\tau,w)$ (see Algorithm \ref{algo:Grobner}), the computation is inconclusive for 1387 over the 2261 inequalities from Step 4.
It decides the birationality for  461 over the 462 irredundant inequalities. 
It rejects 413 over the 1798 redundant inequalities. The computation much likely finishes in the birational cases since the Gröbner basis is then simple.

We also describe a faster deterministic method (called {\it linear triangular}) which allows you to select, 
among the pairs $(\tau,w)$ such that $\pi$ is dominant, some for which $\pi$ is birational.
For Kron$(5,5,5)$, this second method tests 2261 inequalities in 50s and selects 51 of them (among the 462 birational cases).

\subsection{The fibers of \texorpdfstring{$\pi$}{pi}}

Here, we come back to the general context of a reductive group $G$
acting on a representation $\VV$. 
Fix an inequality $(\tau,w)$ and the corresponding map $\pi$. 

\begin{lemma}
  \label{lem:fiber}
  Assume that $\pi$ is dominant and that the dimensions of domain and codomain agree.
  \begin{enumerate}
  \item The map $\pi$ is birational if and only if the fiber $\pi\inv (x)$
    over a general point $x\in \VV^{\tau\leq 0}$ is reduced to
    $\{(e,x)\}$.%, if and only if $\pi'$ is birational where
%\[\pi':
  \item Let $x\in \VV^{\tau\leq 0}$. Then
    $$
    \pi\inv(x)\simeq\{g\in w\inv U^-w\cap U\;:\;gx\in \VV^{\tau\leq
      0}\}=:F_x.
    $$
  \end{enumerate}
\end{lemma}

\begin{proof}
  The map $\pi$ is birational if and only if the fiber $\pi\inv (x)$
    over a general point $x\in \VV$ is reduced to one point. By
    $(w\inv U^-w\cap U)$-equivariance and dominancy of $\pi$, one may
    assume that $x$ belongs to $\VV^{\tau\leq 0}$.
    But, in this case, $\pi\inv(x)$ contains $(e,x)$. This first
    assertion follows.

    Given $x\in \VV^{\tau\leq 0}$, the isomorphism of the second assertion is
    $$
    \begin{array}{ccl}
      \pi\inv(x)&\longto&F_x\\
      (g,y)&\longmapsto &g\inv
    \end{array}
    $$
    with converse $g\longmapsto (g\inv,gx)$.
\end{proof}

Consider an isomorphism $\CC^{\Phi(w)}\cong \mathfrak u(\Phi(w))\stackrel{\varphi}{\rightarrow}w\inv U^-w\cap U$
and fix a basis $(\xi_\chi^*)_{\chi\in \Wt(\VV)}$ of $\VV$ indexed by $\chi\in\Wt(\VV)$ consisting in
$T$-eigenvectors.
Let $(\xi_\chi)_{\chi\in \Wt(\VV)}$ be the dual basis.
Fix $x\in \VV^{\tau\leq 0}$. 

Thanks to Lemma \ref{lem:fiber}, $F_x$ identifies (even scheme-theoretically) with
\begin{equation}\label{eq:eqs_fiber}
\{(v_\beta)_{\beta\in\Phi(w)}\;:\;\xi_\chi(\varphi(v_\beta)x)=0\quad\forall
\chi\in\Wt(\VV)\mbox{ s.t. } \langle\chi,\tau\rangle>0\}.
\end{equation}

Denote by  $f_\chi^x$ the polynomial $(v_\beta)\mapsto
\xi_\chi(\varphi(v_\beta)x)$.
Observe that the functions $f_\chi^x$ vanish at $0$, since
$x\in\VV^{\tau\leq 0}$.

 \subsection{Filter 2: Gröbner bases}
\label{sec:Grobner}
Like we already observed in \eqref{eq:eqs_fiber}, the fiber $F_x$ of Lemma~\ref{lem:fiber} is described as a subscheme of a vector space by an explicit list of equality. This fiber is the reduced point $0$ if and only if the ideal generated by the $(f_{\chi}^x)_{\chi}$ in $\CC[\CC^{\Phi(w)}]$ is the maximal ideal generated by the coordinate functions ($\sharp\Phi(w)$ functions of degree $1$) . This means that the associated Gröbner basis is this family of coordinate functions for any monomial order. Note that if $x$ is chosen in $V_{\QQ}^{\tau\leq 0}$, all the equations lie in $\QQ[\QQ^{\Phi(w)}]$.

However, the computation of Gröbner basis might be too much (time or memory)-consuming. One way to avoid this is to set a time limit (\emph{e.g.} 1s) of computation per pair $(\tau,w)$. This turns this filter into a partial one, leaving some pairs undecided. 

So here is a probabilistic way to implement the search of birationality:

\begin{algorithm}[H]
\algblock[Name]{try}{EndTry}
\algtext*{EndTry}%
    \caption{Deciding birationality of $\pi$ using Gröbner bases (probabilistic answer):}
    \label{algo:Grobner}
    \begin{algorithmic}[1]
\State{Pick a random (general) point $x$ in $\VV_{\QQ}^{\tau\leq0}$.} 
\Statex\Comment{\emph{e.g.} $x$ chosen with integer coordinates in $[\![-1000,1000]\!]$}
\try{ Gr=Gröbner basis of $(f_{\chi}^x)_{\chi}$ in $\QQ[\QQ^{\Phi(w)}]$ \Comment{within a time limit}
   }\label{line:try}
 \EndTry
 \If{line \ref{line:try} finished}
 \If{Gr is made of $\sharp\Phi(w)$ functions of degree $1$}
 \Return{True} 
 \Statex\Comment{Then $\pi$ is birational without uncertainty}
 \Else{}
 \Return{False} 
 \Statex\Comment{even if $\pi$ might be birational with very low probability (if $x$ was not general enough)}
 \EndIf
 \Else{}
 \Return{inconclusive computation}
 \EndIf
   \end{algorithmic}
\end{algorithm}

In order to turn the algorithm into a deterministic one, we only have to replace the random $x$ by the generic point. In other words, let $x$ be the symbolic vector $\sum_{\chi\, :\,\langle \chi, \tau\rangle\leq 0} x_{\chi}\xi^*_{\chi}\in  \QQ(\VV_{\QQ}^{\tau\leq 0})\cong \QQ(x_{\chi}\,:\,\langle \chi, \tau\rangle\leq 0)$
The equations \eqref{eq:eqs_fiber} are now equations $(f_{\chi}^x)_{\chi}$ in $\QQ(\VV_{\QQ}^{\tau\leq 0})[\QQ^{\Phi(w)}]$, and we can compute the Gröbner basis with coefficients in $\QQ(\VV_{\QQ}^{\tau\leq 0})$.
This method is deterministic but slower. For instance, we need a 5-second limit to positively decide the 47 birational cases in Kron(4,4,4) where a 1-second limit was enough with the probabilistic method.

\subsection{Filter 3: Linear Triangular}
\label{sec:LinTri}
In some cases, there is no need to use the heavy machinery of Gröbner basis. Let us first describe an inductive procedure:

\begin{algorithm}[H]
    \caption{A recursive procedure:}
    \label{algo:LinTri_rcurs}
    \begin{algorithmic}[1]
\State{Let $S=\{\chi\in \Wt^{\tau>0} \;:\; f_\chi^x(\in \QQ(\VV_{\QQ}^{\tau\leq 0})[\QQ^{\Phi(w)}]) 
\textrm{ is linear on } \QQ^{\Phi(w)}\}$.} \label{line:select_lin}
\State{Let $\Phi_1=\{\beta\in \Phi(w)\; :\; \exists \chi\in S \textrm{ s.t. }
  v_\beta \textrm{ occurs nontrivially in} f_\chi^x\}$}.
\If {$\Phi_1$ and $S$ share the same cardinality}
\State{Plugin $v_\beta=0$ for all $\beta\in \Phi_1$ in the
    equations $f_\chi^x$ for $\chi\not\in S$.}
\State{Restart at line~\ref{line:select_lin} with equations indexed by $\Wt^{\tau>0}-S$ and with $\QQ^{\Phi(w)}$ replaced by  
$\QQ^{\Phi(w)-\Phi_1}$.}
\EndIf
   \end{algorithmic}
\end{algorithm}
    
If this inductive procedure exhausts $\Phi(w)$,
we say that the system of equations $f_\chi^x=0$ (or the pair $(\tau,w)$) is {\it linear triangular} in the unknown indexed by $\Phi(w)$.

\begin{lemma}
Under the assumption of Lemma \ref{lem:fiber}, if the system of equations $(f_\chi^x)_{\chi \in \Wt^{\tau>0}}=0$ is linear triangular then $\pi$ is birational.
\end{lemma}

\begin{proof}
By construction, the polynomial $f_\chi^x\in \QQ[\VV_{\QQ}^{\tau\leq 0}][\QQ^{\Phi(w)}]$ identifies with $\xi_{\chi}\circ\pi$. 
Since $\VV$ is a linear representation, $f_\chi^x$ is of degree one in $\VV_{\QQ}^{\tau\leq 0}$. 
When it is also linear in the other variables, it thus identifies with $\xi_{\chi}\circ T_{(e,x)}\pi$. 

Let $S, \Phi_1$ as in Algorithm \ref{algo:LinTri_rcurs}, so that $(f_\chi^x)_{\chi \in S}=0$ 
is a linear system with the same number of equations and of unknowns (indexed by $\Phi_1$). 
From the discussion obove, it identifies with a square block of the block-triangular system $T_{(e,x)}\pi=0$.
By the hypothesis on $\pi$, its only solution is zero. 
Whence the inductive step.
\end{proof}

\bibliographystyle{alpha}
\bibliography{biblio_BDR}
\end{document}